\documentclass[11pt]{article}
\usepackage[tbtags]{amsmath}
\usepackage{amssymb}
\usepackage{amsthm}
\usepackage[misc]{ifsym}
\usepackage{cases}
\usepackage{mathrsfs}
\usepackage{color}
\usepackage{hyperref}
\usepackage{graphicx}
\usepackage{float}
\numberwithin{equation}{section}
\normalsize

\title{\bf Indefinite Linear-Quadratic Optimal Control Problems of Backward Stochastic Differential Equations with Partial Information}
\author{\normalsize 
		Jialong Li\thanks{\it School of Mathematics, Shandong University, Jinan 250100, P.R. China, E-mail: lijialong@mail.sdu.edu.cn},\quad
		Zhiyong Yu\thanks{\it  School of Mathematics, Shandong University, Jinan 250100, P.R. China, E-mail: yuzhiyong@sdu.edu.cn},\quad 
		Wanying Yue\thanks{\it Corresponding author. School of Mathematics, Shandong University, Jinan 250100, P.R. China, E-mail: yuewanying@mail.sdu.edu.cn}
		}
\newtheorem{Proposition}{Proposition}[section]
\newtheorem{Theorem}{Theorem}[section]

\newtheorem{Lemma}{Lemma}[section]
\newtheorem{Remark}{Remark}[section]
\newtheorem{Example}{Example}[section]
\newtheorem{Corollary}{Corollary}[section]

\begin{document}

\maketitle

\noindent{\bf Abstract:}\quad
This paper is concerned with a kind of linear-quadratic (LQ) optimal control problem of backward stochastic differential equation (BSDE) with partial information.
The cost functional includes cross terms between the state and control, and the weighting matrices are allowed to be indefinite.
Through variational methods and stochastic filtering techniques, we derive the necessary and sufficient conditions 
for the optimal control, where a Hamiltonian system plays a crucial role. Moreover, to construct the optimal control, we introduce a matrix-valued differential equation and a BSDE with filtering, and establish their solvability under the assumption that the cost functional is uniformly convex.
Finally, we present explicit forms of the optimal control and value function.\vspace{2mm}

\vspace{2mm}


\section{Introduction}
The LQ optimal control problem is of great importance in both theory and practice. 
Compared with general control problems, it has a more concise form, making it easier to obtain favorable results.
The control system can be either deterministic (see \cite{LQ1}) or stochastic (see \cite{LQ2,LQ3}). 
The key to LQ problem is to obtain the feedback form of the optimal control, which often involves a specific type of ordinary differential equation (ODE) known as Riccati equation.
Readers may refer to Yong and Zhou~\cite{LQ4} for more details on the forward stochastic linear quadratic (FSLQ) control problem.
With the development of BSDE theory, research on control problems involving BSDEs has become increasingly prevalent (see, e.g., \cite{BLQ2, BLQ3, BLQ4}).
As for the backward stochastic linear quadratic (BSLQ) optimal control problem, it was first solved by Lim and Zhou~\cite{BLQ5}. 
They provided a feedback form of the optimal control under the conditions that the coefficients are deterministic and positive semi-definite.
Li et al.~\cite{BLQ6} extended the BSLQ problem to the mean-field case, while Sun and Wang~\cite{BLQ7} considered the case with stochastic coefficients.

In reality, it is often not possible to observe the complete information of the system. For instance, in financial markets, investors may not know exactly all the factors that affect asset prices. 
Such control problems are referred to as stochastic optimal control problems with incomplete information, which usually consist of two components: filtering and control. As Wang et al.~\cite{PLQ10} noted, incomplete information can 
generally be classified into two types: partial observation and partial information. In the case of partial observation, the available information is characterized by the filtration generated by the observation process, 
and in general, the observation process often depends on the control process. However, under partial information, the available information is a subfiltration of the complete information, which is abstract and does not depend on the control process. 
Nagai and Peng~\cite{PLQ1}, Xiong and Zhou~\cite{PLQ2} investigated portfolio optimization problems under incomplete information.
Hu and Øksendal~\cite{PLQ3} studied a stochastic LQ optimal control problem with jumps under incomplete information.
Meng~\cite{PLQ4} obtained a maximum principle and a verification theorem for a incomplete information stochastic optimal control, the controlled system of which is
a fully coupled nonlinear forward-backward stochastic differential equation (FBSDE).
Wang et al.~\cite{PLQ5} researched a linear FBSDE system with incomplete information. Through the combination of a backward separation approach, classical variational method, and stochastic filtering,
they derived two optimality conditions and an explicit representation of the optimal control.
Huang et al.~\cite{PLQ6} explored backward mean-field LQ games with both complete and incomplete information.
Huang et al.~\cite{PLQ7} and Wang et al.~\cite{PLQ8} investigated BSLQ problems under partial information and obtained the feedback representations. Yang et al.~\cite{PLQ9} studied a mean-field stochastic LQ problem with jumps under incomplete information.
For further understanding of incomplete information control problems, interested readers may refer to Wang et al.~\cite{PLQ10}.

It should be noted that the aforementioned literature concerning LQ problems generally assumes the nonnegative definiteness of the weighting matrices in the cost functional.
In this paper, we are interested in how the indefiniteness of weighting matrices influences the BSLQ problem with partial information.
The study of indefinite LQ problems can be traced back to Peng \cite{P92} and Chen et al.~\cite{ILQ1}, who pointed out that the nonnegative definiteness of the weighting matrices is not a necessary requirement in stochastic LQ problem.
Subsequently, Ait Rami et al.~\cite{ILQ2} studied an indefinite FSLQ problem
and proposed a generalized Riccati equation. They revealed that the solvability of the generalized Riccati equation is not only sufficient
but also necessary for the well-posedness of the indefinite FSLQ problem and the existence of the optimal control.
Ni et al.~\cite{ILQ3, ILQ4} studied indefinite mean-field FSLQ problems in discrete case, including both finite and infinite time horizons.
Sun et al.~\cite{ILQ5, ILQ6}, using the uniform convexity of the cost functional, investigated the open-loop solvability of indefinite mean-field FSLQ problems and indefinite BSLQ problems, respectively.
However, there has been little research on indefinite partial information LQ problems.
In a recent paper, Li et al.~\cite{ILQ7} explored the weak closed-loop solvability of indefinite FSLQ problems with partial information. 
Unlike the complete information case, they introduced two Riccati equations and, using
a perturbation approach, obtained the open-loop solvability of the problem. By contrast, the structures of backward systems are fundamentally different from those of forward systems. To the best of our knowledge, research on indefinite  BSLQ problems with partial information is lacking, and we aim to fill this gap.

In this study, we focus on the partial information BSLQ problem with indefinite cost weighting matrices. Our main contributions and differences from the existing literature can be summarized as follows.


(1) We establish the solvability of the BSLQ optimal control problem with indefinite cost weighting matrices under partial information. 
The first challenge posed by indefiniteness is that we cannot be certain whether a solution to the control problem exists. 
To overcome this challenge, inspired by \cite{ILQ6}, we assume that the cost functional is uniformly convex.
Unlike \cite{PLQ7, PLQ8}, which require positive semi-definiteness of the weighting matrices, we remove this restriction, substantially increasing analytical difficulty but enhancing practical significance. 
Moreover, in contrast to the work in \cite{ILQ7} on indefinite FSLQ problems with partial information, we focus on backward systems, whose structure is fundamentally different from that of forward problems, and therefore different analytical techniques are required.


(2) We extend the connection between the BSLQ problem and its corresponding FSLQ problem under complete information, 
established in \cite{BLQ5} and \cite{ILQ6}, to the framework of partial information. 
This connection is essential for proving the solvability of the matrix-valued differential equation introduced in the construction of the optimal control.

(3) With the help of the Lyapunov and Riccati equations introduced in the forward problem with partial information,
we establish the solvability of the matrix-valued differential equation for the backward problem through a limiting procedure, 
thereby constructing an explicit form of the optimal control. Compared with \cite{ILQ7}, 
our stochastic Hamiltonian system is formulated as an FBSDE that is coupled at the initial condition and incorporates filtering, 
which makes the associated matrix-valued differential equation more intricate and harder to solve. 
Unlike \cite{ILQ6}, under the framework of partial information, 
the construction of the optimal control additionally requires a Lyapunov equation and a BSDE with filtering.

(4) The cross terms of state $Y$ with $Z_1$, $Z_2$, and control $u$ are included in the cost functional (see \eqref{BSLQ:CF}).
In previous studies on partial information BSLQ problems with positive semi-definite weighting matrices, for example \cite{PLQ7, PLQ8}, cross terms are generally not included.
The presence of cross terms increases the generality of the problem 
while introducing additional complexity, especially in the process of constructing the optimal control. 

The rest of this paper is organized as follows. In Section~\ref{sect2}, the indefinite BSLQ problem with partial information is formulated.
In Section~\ref{sec3}, we first explore the relationship between the BSLQ problem and its corresponding forward problem. 
We then apply a simplification method to the cost functional and construct a Hamiltonian system.
To decouple the system, we introduce a matrix-valued differential equation and a BSDE with filtering.
After that, we provide explicit forms of the optimal control and value function. Section~\ref{sec4} concludes the paper.

\section{Problem Formulation} \label{sect2}
Throughout this paper, let $\mathbb{R}^n$ denote the $n$-dimensional
Euclidean space and $\mathbb{R}^{m\times n}$ denote the the space of all ($m\times n$) matrices,
equipped with the inner product $\langle M,\, N\rangle = \mbox{tr}\,(M^{\top}N)$ and the induced norm $|M| = \sqrt{\mbox{tr}\,(M^{\top}M)}$, where the superscript $\top$ denotes the transpose of vectors or matrices.
When there is no ambiguity, we also use $\langle \cdot,\, \cdot \rangle$ to denote inner products in other spaces.
In particular, we use $\mathbb{S}^n$ to denote the space of ($n \times n$) symmetric matrices, 
and $\mathbb{S}_{+}^n$ (resp., $\widehat{\mathbb{S}}_{+}^n$) to represent the set of ($n \times n$) positive semi-definite (resp., positive definite) symmetric matrices.
$I_n$ denotes the ($n \times n$) identity matrix. We often omit the index $n$ when it is clear from the context.  
For matrices $M, N\in\mathbb{S}^n$, we write $M \geqslant N$ (resp., $M > N$) if $M - N$ is positive semi-definite (resp., positive definite). 
Let $T>0$ be a fixed time horizon. For a mapping $S: [0,T]\rightarrow\mathbb{S}^n$, we denote $S(\cdot) \geq 0$ (resp., $>0$) if $S(t) \geq 0$ (resp., $>0$) for all $t\in[0,T]$, and $S(\cdot) \gg 0$ if there exists a constant $\delta > 0$ such that $S(t) \geqslant \delta I_n$ for all $t\in[0,T]$.

Let $(\Omega, \mathcal{F}, \mathbb{F}, \mathbb{P})$ be a complete filtered probability space on which a two-dimensional
standard Brownian motion $(W_1(\cdot), W_2(\cdot))^{\top}$ is defined. $\mathbb{F} = \{\mathcal{F}_t\}_{t \in [0,T]}$ is the natural filtration of $W_1(\cdot)$ and $W_2(\cdot)$ augmented by all 
$\mathbb{P}$-null sets in $\mathcal{F}$, where $\mathcal{F}=\mathcal{F}_T$.
Let $\mathbb{G} = \{\mathcal{G}_t\}_{t\in [0,T]}$ be the natural filtration of $W_2(\cdot)$ augmented by all $\mathbb{P}$-null sets in $\mathcal{G}$, where $\mathcal{G}=\mathcal{G}_T$.
In this paper, let $\mathcal{G}_t$ represent the information available at time $t$. Obviously, $\mathbb{G}$ is a subfiltration of $\mathbb{F}$.
For any Banach space $\mathbb{H}$, we adopt the following notations:
\[
	\begin{aligned}
		& C([0,T]; \mathbb{H}) = \left\{f: [0,T] \rightarrow \mathbb{H} \mid f \text { is continuous}\right\}, \\
		& L^{\infty}(0,T; \mathbb{H}) = \left\{f: [0,T] \rightarrow \mathbb{H} \mid f \text { is Lebesgue measurable and essentially bounded}\right\}, \\
		& L_{\mathcal{F}_T}^2(\Omega; \mathbb{H}) = \left\{\xi: \Omega \rightarrow \mathbb{H} \mid \xi \text { is } \mathcal{F}_T \text {-measurable, } \mathbb{E}|\xi|^2<\infty\right\}, \\
		& L_{\mathbb{F}}^2(0, T; \mathbb{H}) = \left\{f:[0, T] \times \Omega \rightarrow \mathbb{H} \mid f \text { is } \mathbb{F} \text {-progressively measurable, } \mathbb{E} \int_0^T|f(t)|^2\, \mathrm{d}t<\infty\right\}, \\
		& S_{\mathbb{F}}^2(0, T; \mathbb{H})  = \left\{f:[0, T] \times \Omega \rightarrow \mathbb{H} \mid f \text { is } \mathbb{F} \text {-progressively measurable, continuous,} \right. \\
&\qquad\qquad\qquad\left. \mathbb{E}\left[\sup _{0 \leq t \leq T}|f(t)|^2\right]<\infty\right\}.
	\end{aligned}
\]
The space $L_{\mathbb{G}}^2(0, T; \mathbb{H})$ and $S_{\mathbb{G}}^2(0, T; \mathbb{H})$ can be defined in a similar manner.
Moreover, for any $\mathbb{F}$-progressively measurable stochastic process $f(\cdot)$, let 
\[
  \widehat{f}(t) = \mathbb{E}[f(t)\mid\mathcal{G}_t]
\]
denote the optimal filter with respect to $\mathcal{G}_t$ for any $t\in [0,T]$.

Now consider the BSDE
\begin{equation} \label{BSLQ:SE}
	\begin{cases}
		\begin{aligned}
		\mathrm{d}Y(t) &= 
		\big[ A(t)Y(t) + B(t)u(t) + C_{1}(t)Z_{1}(t) + C_{2}(t)Z_{2}(t) \big]\, \mathrm{d}t \\
		&\quad + Z_{1}(t)\, \mathrm{d}W_{1}(t) + Z_{2}(t)\, \mathrm{d}W_{2}(t), \\
		Y(T) &= \xi ,
		\end{aligned}
	\end{cases}
\end{equation}
where $u(\cdot)$ is the control process; $A(\cdot)$, $B(\cdot)$, $C_1(\cdot)$, $C_2(\cdot)$ are deterministic matrix-valued functions of proper dimensions;
$\xi\in L_{\mathcal{F}_T}^2(\Omega; \mathbb{R}^n)$ is the terminal state. The admissible control set is
\[
	\mathcal{U}_{ad} = L_{\mathbb{G}}^2(0, T; \mathbb{R}^m).
\]
Any $u(\cdot)\in \mathcal{U}_{ad}$ is called an admissible control.

\bigskip

\noindent{\bf Hypothesis (H1)} \ The coefficients of the state equation~\eqref{BSLQ:SE} satisfy the following:
\[
\begin{cases}
	A(\cdot) \in L^\infty\left(0, T ; \mathbb{R}^{n \times n}\right), \\
	B(\cdot) \in L^\infty\left(0, T ; \mathbb{R}^{n \times m}\right), \\
	C_{1}(\cdot), C_{2}(\cdot) \in L^\infty\left(0, T ; \mathbb{R}^{n \times n}\right).
\end{cases}
\]
By the classical results of BSDEs (see \cite{BLQ1}), under (H1), for any $\xi\in L_{\mathcal{F}_T}^2(\Omega; \mathbb{R}^n)$ and $u(\cdot) \in \mathcal{U}_{ad}$, 
the state equation~\eqref{BSLQ:SE} admits a unique solution $(Y(\cdot),Z_1(\cdot),Z_2(\cdot))\in S_{\mathbb{F}}^2(0, T; \mathbb{R}^n)\times L_{\mathbb{F}}^2(0, T; \mathbb{R}^n)\times L_{\mathbb{F}}^2(0, T; \mathbb{R}^n)$, 
which is called the state process corresponding to the control $u(\cdot)$.

We then introduce the following cost functional
\begin{equation} \label{BSLQ:CF}
	\begin{aligned}
	J(\xi; u(\cdot)) = &\ \mathbb{E} \bigg\{
		\big\langle GY(0),\, Y(0) \big\rangle
		+\int_0^T 
		\Big[
		\big\langle Q(t)Y(t),\, Y(t) \big\rangle
		+2\big\langle S_1(t)Y(t),\, Z_1(t) \big\rangle \\
		&+2\big\langle S_2(t)Y(t),\, Z_2(t) \big\rangle
		+2\big\langle S_3(t)Y(t),\, u(t) \big\rangle 
		+\big\langle N_1(t) Z_1(t),\, Z_1(t)\big\rangle \\
		&+\big\langle N_2(t) Z_2(t),\, Z_2(t) \big\rangle
		+\big\langle R(t)u(t),\, u(t) \big\rangle
		\Big]
		\,\mathrm{d}t
	\bigg\}.
	\end{aligned}
\end{equation}

\noindent{\bf Hypothesis (H2)} \ 
The weighting matrices in the cost functional \eqref{BSLQ:CF} satisfy the following:
\[
\begin{cases}
	G \in \mathbb{S}^n,\quad Q(\cdot), N_1(\cdot), N_2(\cdot) \in L^\infty\left(0, T ; \mathbb{S}^{n}\right), \quad R(\cdot) \in L^\infty\left(0, T ; \mathbb{S}^{m}\right),\\
	S_1(\cdot), S_2(\cdot) \in L^\infty\big(0, T ; \mathbb{R}^{n \times n}\big),\quad S_3(\cdot) \in L^\infty\big(0, T ; \mathbb{R}^{m \times n}\big).
\end{cases}
\]
Under (H1) and (H2), for any $\xi\in L_{\mathcal{F}_T}^2(\Omega; \mathbb{R}^n)$ and $u(\cdot) \in \mathcal{U}_{ad}$, the cost functional~\eqref{BSLQ:CF} is well-defined.
Assumptions (H1) and (H2) impose boundedness on the coefficients, which will be frequently used in the subsequent proofs.
It should be noted that, in our indefinite control problem, the coefficients in the cost functional are not necessarily positive semi-definite.
Our BSLQ control problem with partial information can be stated as follows. 

\bigskip

\noindent{\bf Problem (BSLQ-P)} \
	For a given terminal state $\xi\in L_{\mathcal{F}_T}^2(\Omega; \mathbb{R}^n)$, find a control $u^*(\cdot)\in\mathcal{U}_{ad}$ such that
	\begin{equation} \label{BSLQ:target}
	J(\xi ; u^*(\cdot)) = \inf_{u(\cdot) \in \mathcal{U}_{ad}} J(\xi ; u(\cdot)) =: V(\xi).
	\end{equation}	
Any $u^*(\cdot) \in \mathcal{U}_{ad}$ satisfying \eqref{BSLQ:target} is called an optimal control, and the corresponding $\linebreak(Y^*(\cdot),Z_1^*(\cdot),Z_2^*(\cdot))$
is called an optimal state. $V(\xi)$ is called the value function of Problem (BSLQ-P).

The above boundedness assumption alone is not enough for solving the problem. 
We need to impose slightly stronger conditions.
Sun et al.~\cite{ILQ6} studied a BSLQ problem with complete information from a Hilbert space point of view.
They found that, as long as the optimal control exists, one can use a limiting procedure to approach it,
where the key is to solve the control problem under the uniform convexity condition.
In fact, the case under partial information is similar. The main difference lies in the change of the admissible control set.
Interested readers may refer to \cite[Section~3]{ILQ6} for more information.
We now introduce the third assumption. 
\bigskip

\noindent{\bf Hypothesis (H3)} \ There exists a $\delta > 0$ such that 
\[
	J(0;u(\cdot)) \geqslant \delta\, \mathbb{E} \int_{0}^{T}\left\lvert u(t)\right\rvert^2\, \mathrm{d}t, \quad \text{for all } u(\cdot) \in \mathcal{U}_{ad}.
\]
This assumption is called the uniform convexity condition of the cost functional.
As will be seen in the subsequent analysis, it guarantees the existence and uniqueness of the optimal control for Problem (BSLQ-P)
and plays a key role in exploring the connection between backward and forward problems in Section~\ref{sec3.1}.

At the end of this section, we present an illustrative example to show that the uniform convexity condition of the cost functional allows the weighting matrices to be indefinite.
\begin{Example}\label{Exa}
Consider a control problem with the following one-dimensional state equation
\[
	\begin{cases}
		\begin{aligned}
		\mathrm{d}Y(t) &= 
		\big[Y(t) + u(t)\big]\, \mathrm{d}t + Z_{1}(t)\, \mathrm{d}W_{1}(t) + Z_{2}(t)\, \mathrm{d}W_{2}(t), \\
		Y(1) &= \xi
		\end{aligned}
	\end{cases}
\]
and the cost functional
\[
	\begin{aligned}
	J(\xi; u(\cdot)) = &\ \mathbb{E}\left\{-|Y(0)|^2+
		\int_0^1 
		\big[
		5|u(t)|^2 -|Y(t)|^2- |Z_1(t)|^2 - |Z_2(t)|^2
		\big]
		\,\mathrm{d}t\right\}.
	\end{aligned}
\]
Obviously, the weighting matrices are not all nonnegative. We now verify that the cost functional satisfies (H3).

In fact, setting $\xi=0$ and applying It\^o's formula to $|Y(\cdot)|^2$, we obtain
\[
\begin{aligned}
&\mathbb E\left\{|Y(0)|^2+\int_0^1\big[2|Y(t)|^2+|Z_1(t)|^2+|Z_2(t)|^2\big]\, \mathrm dt\right\}=-\mathbb E\int_0^12Y(t)u(t)\, \mathrm dt\\
&\leq \mathbb E\int_0^12|Y(t)||u(t)|\, \mathrm dt \leq \mathbb E\int_0^1\big[|Y(t)|^2+|u(t)|^2\big]\, \mathrm dt,
\end{aligned}
\]
which immediately yields
\[
\begin{aligned}
&\mathbb E\left\{|Y(0)|^2+\int_0^1\big[|Y(t)|^2+|Z_1(t)|^2+|Z_2(t)|^2\big]\, \mathrm dt\right\} \leq \mathbb E\int_0^1|u(t)|^2\, \mathrm dt.
\end{aligned}
\]
Thus, we have
\[
J(0;u(\cdot)) \geq 4\mathbb E\int_0^1|u(t)|^2\,\mathrm dt, \quad \text{for all } u(\cdot) \in \mathcal{U}_{ad},
\]
which implies that the uniform convexity condition holds.
\end{Example}

\section{Main Results} \label{sec3}
\subsection{Connection with FSLQ problems with partial information} \label{sec3.1}
Before proceeding further with Problem (BSLQ-P), we first examine the forward case.
This section focuses on the relationship between the backward and forward problems under assumption (H3). From the analysis, we derive some useful results that not only reveal properties of the weighting coefficients in Problem (BSLQ-P),
but also play a crucial role in proving the unique solvability of the matrix-valued differential equation in subsequent sections.

Consider the stochastic differential equation (SDE)
\[
	\begin{cases}
	\begin{aligned}
		\mathrm{d}\mathcal{X}(t) &= 
		\big[ \mathcal{A}(t)\mathcal{X}(t) + \mathcal{B}(t)v(t) \big]\, \mathrm{d}t
		+\big[ \mathcal{C}_1(t)\mathcal{X}(t) + \mathcal{D}_1(t)v(t) \big]\, \mathrm{d}W_1(t) \\
  		&\quad+\big[ \mathcal{C}_2(t)\mathcal{X}(t) + \mathcal{D}_2(t)v(t) \big]\, \mathrm{d}W_2(t),\\
		\mathcal{X}(0) &= x,
	\end{aligned}
	\end{cases}
\]
and the cost functional
\[
	\begin{aligned}
		\mathcal{J}(x; v(\cdot)) &=  \mathbb{E} \bigg\{
		\big\langle \mathcal{H}\mathcal{X}(T),\, \mathcal{X}(T) \big\rangle 
		+\int_0^T \Big[\big\langle \mathcal{Q}(t)\mathcal{X}(t), \, \mathcal{X}(t) \big\rangle \\
		&\quad + 2\big\langle \mathcal{S}(t)\mathcal{X}(t), \, v(t) \big\rangle 
		+ \big\langle \mathcal{R}(t)v(t), \, v(t) \big\rangle\Big]\,
		\mathrm{d}t
		\bigg\},
	\end{aligned}
\]
where the cofficients satisfy
\[	
	\begin{cases}
	\mathcal{A}(\cdot),\mathcal{C}_{1}(\cdot), \mathcal{C}_{2}(\cdot) \in L^\infty\left(0, T ; \mathbb{R}^{n \times n}\right), \\
	\mathcal{B}(\cdot), \mathcal{D}_1(\cdot), \mathcal{D}_2(\cdot) \in L^\infty\left(0, T ; \mathbb{R}^{n \times m}\right), \\
	\mathcal{H} \in \mathbb{S}^n, \quad \mathcal{Q}(\cdot) \in L^\infty\left(0, T ; \mathbb{S}^{n}\right), \\
	\mathcal{S}(\cdot) \in L^\infty\big(0, T ; \mathbb{R}^{m \times n}\big), \quad \mathcal{R}(\cdot) \in L^\infty\left(0, T ; \mathbb{S}^m\right).
	\end{cases}
\]
The FSLQ optimal control problem with partial information is formulated as follows. 

\bigskip

\noindent{\bf Problem (FSLQ-P)} \
For a given initial state $x \in \mathbb{R}^n$, find a control $v^*(\cdot) \in \mathcal{U}_{ad}$ such that
\[
	\mathcal{J}(x ; v^*(\cdot)) = \inf_{v(\cdot) \in \mathcal{U}_{ad}} \mathcal{J}(x;v(\cdot)) =: \mathcal{V}(x).
\]

In the rest of this paper, we may suppress $t$ for the simplicity of notations when no confusion arises.
There are a Lyapunov equation and a Riccati equation closely related to Problem (FSLQ-P):
\begin{equation} \label{FSLQ:RE1}
	\begin{cases}
	\dot{\mathcal{P}_1} + \mathcal{P}_1\mathcal{A} + \mathcal{A}^{\top}\mathcal{P}_1 + \mathcal{C}_1^{\top}\mathcal{P}_1\mathcal{C}_1 + \mathcal{C}_2^{\top}\mathcal{P}_1\mathcal{C}_2 + \mathcal{Q}  = 0,\\
	\mathcal{P}_1(T) = \mathcal{H},
	\end{cases}
\end{equation}
\begin{equation} \label{FSLQ:RE2}
	\begin{cases}
	\dot{\mathcal{P}_2} + \mathcal{P}_2\mathcal{A} + \mathcal{A}^{\top}\mathcal{P}_2 + \mathcal{C}_1^{\top}\mathcal{P}_1\mathcal{C}_1 + \mathcal{C}_2^{\top}\mathcal{P}_2\mathcal{C}_2 \\
	\quad \,\, -\left(\mathcal{B}^{\top}\mathcal{P}_2 + \mathcal{D}_1^{\top}\mathcal{P}_1\mathcal{C}_1 + \mathcal{D}_2^{\top}\mathcal{P}_2\mathcal{C}_2 + \mathcal{S}\right)^{\top}\left(\mathcal{R} + \mathcal{D}_1^{\top}\mathcal{P}_1\mathcal{D}_1 + \mathcal{D}_2^{\top}\mathcal{P}_2\mathcal{D}_2\right)^{-1} \\
	\quad \,\, \times \left(\mathcal{B}^{\top}\mathcal{P}_2 + \mathcal{D}_1^{\top}\mathcal{P}_1\mathcal{C}_1 + \mathcal{D}_2^{\top}\mathcal{P}_2\mathcal{C}_2 + \mathcal{S}\right) + \mathcal{Q} = 0, \\
	\mathcal{P}_2(T) = \mathcal{H},
	\end{cases}
\end{equation}
where the invertibility of $\mathcal{R} + \mathcal{D}_1^{\top}\mathcal{P}_1\mathcal{D}_1 + \mathcal{D}_2^{\top}\mathcal{P}_2\mathcal{D}_2$ will be established in the following Lemma \ref{lemma:convexity}.

The following lemmas ensure the solvability of the two equations and Problem (FSLQ-P). These results have been established in the existing literature (see, e.g., \cite{ILQ6,ILQ7}).

\begin{Lemma} \label{lemma:Req1}
Lyapunov equation~\eqref{FSLQ:RE1} admits a unique solution $\mathcal{P}_1(\cdot) \in C\left([0,T];\mathbb{S}^n\right)$.
In addition, if $\mathcal{H} \geqslant 0$ (resp., $\mathcal{H}>0$) and $\mathcal{Q}(\cdot) \geqslant 0$, then $\mathcal{P}_1(\cdot) \geqslant 0$ (resp., $\mathcal{P}_1(\cdot) > 0$).
\end{Lemma}

\begin{Lemma} \label{lemma:convexity}
Suppose that there exists a constant $\alpha > 0$ such that
\begin{equation} \label{FSLQ:UC}
	\mathcal{J}(0;v(\cdot)) \geqslant \alpha\, \mathbb{E}\int_{0}^{T}\big|v(t)\big|^2\, \mathrm{d}t, \quad \text{for all } v(\cdot) \in \mathcal{U}_{ad}.
\end{equation}
Then Riccati equation~\eqref{FSLQ:RE2} admits a unique solution $\mathcal{P}_2(\cdot) \in C([0,T]; \mathbb{S}^n)$ such that
\[
\mathcal{R} + \mathcal{D}_1^{\top}\mathcal{P}_1\mathcal{D}_1 + \mathcal{D}_2^{\top}\mathcal{P}_2\mathcal{D}_2 \gg 0.
\]
Let
\[
\varTheta^* = -\left(\mathcal{R} + \mathcal{D}_1^{\top}\mathcal{P}_1\mathcal{D}_1 + \mathcal{D}_2^{\top}\mathcal{P}_2\mathcal{D}_2 \right)^{-1}
\left(\mathcal{B}^{\top}\mathcal{P}_2 + \mathcal{D}_1^{\top}\mathcal{P}_1\mathcal{C}_1 + \mathcal{D}_2^{\top}\mathcal{P}_2\mathcal{C}_2 + \mathcal{S}\right),
\]
then Problem (FSLQ-P) has a unique optimal control
\[
v^*(\cdot) = \varTheta^*(\cdot)\widehat{\mathcal{X}}^*(\cdot),
\]
which is a linear feedback of the state filtering estimation.
$\mathcal{P}_1(\cdot)$, $\mathcal{P}_2(\cdot)$ are solutions to Lyapunov equation \eqref{FSLQ:RE1} and Riccati equation \eqref{FSLQ:RE2}, respectively, and the filtering estimate $\widehat{\mathcal{X}}^*(\cdot)$ satisfies
\[
\begin{cases}
	\begin{aligned}
	\mathrm{d}\widehat{\mathcal{X}}^*(t) &= [\mathcal{A}(t)\widehat{\mathcal{X}}^*(t) + \mathcal{B}(t)v^*(t)]\, \mathrm{d}t + [\mathcal{C}_2(t)\widehat{\mathcal{X}}^*(t) + \mathcal{D}_2(t)v^*(t)]\, \mathrm{d}W_2(t),\\ 
	\widehat{\mathcal{X}}^*(0) &= x.
	\end{aligned}
\end{cases}
\]
Moreover, the value function $\mathcal{V}(\cdot)$ is given by
\begin{equation}
	\mathcal{V}(x) = \langle \mathcal{P}_2(0)x,\ x\rangle, \quad \text{for any }x \in \mathbb R^n.
\end{equation}
\end{Lemma}

\begin{Lemma} \label{lemma:std eq2}
	Suppose that
		\begin{equation} \label{FSLQ:STD}
			\mathcal{H} \geqslant 0, \quad \mathcal{R}(\cdot)\gg 0, \quad \mathcal{Q}(\cdot)-\mathcal{S}(\cdot)^{\top}\mathcal{R}(\cdot)^{-1}\mathcal{S}(\cdot) \geqslant 0.\\
		\end{equation}
		Then \eqref{FSLQ:UC} holds for some $\alpha > 0$,
		and the Riccati equation~\eqref{FSLQ:RE2} admits a unique solution $\mathcal P_2(\cdot) \geqslant 0$.
In addition, if $\mathcal{H} > 0$, then  $\mathcal{P}_2(\cdot) > 0$.
\end{Lemma}

Inspired by \cite{BLQ5} and \cite{ILQ6}, let us consider a special FSLQ optimal control problem with partial information,
which is closely related to Problem (BSLQ-P).
The state equation is
\begin{equation} \label{FSLQn:SE}
	\begin{cases}
	\begin{aligned}
		\mathrm{d}X(t) &= 
		\big[ A(t)X(t) + B(t)u(t) + C_{1}(t)v_{1}(t) + C_{2}(t)v_{2}(t) \big]\, \mathrm{d}t \\
		 &\quad+ v_{1}(t)\, \mathrm{d}W_{1}(t) + v_{2}(t)\, \mathrm{d}W_{2}(t), \\ 
		X(0) &= x,
	\end{aligned}
	\end{cases}
\end{equation}
with the initial state $x\in\mathbb{R}^n$, and the cost functional
\begin{equation} \label{FSLQn:CF}
	\begin{aligned}
		\mathcal{J}_\lambda(x; u, v_1, v_2) 
		 =& \ \mathbb{E} \bigg\{
			\lambda |X(T)|^2
			+\int_0^T 
			\Big[\big\langle Q(t)X(t),\, X(t) \big\rangle
			+2\big\langle S_1(t)X(t),\, v_1(t) \big\rangle \\
			&+2\big\langle S_2(t)X(t),\, v_2(t) \big\rangle
			+2\big\langle S_3(t)X(t),\, u(t) \big\rangle
			+\big\langle N_1(t) v_1(t),\, v_1(t)\big\rangle\\
			&+\big\langle N_2(t) v_2(t),\, v_2(t) \big\rangle
			+\big\langle R(t)u(t),\, u(t) \big\rangle\Big]
			\,\mathrm{d}t
		\bigg\},
		\end{aligned}
\end{equation}
where $\lambda > 0$ and other coefficients are the same as Problem (BSLQ-P). In the above system, the control is 
\[
(u,v_1,v_2) \in L^2_{\mathbb{G}}(0,T;\mathbb{R}^m) \times L^2_{\mathbb{G}}(0,T;\mathbb{R}^n) \times L^2_{\mathbb{G}}(0,T;\mathbb{R}^n) =: \widetilde{\mathcal{U}}_{ad}.
\]
The optimal control problem is formulated as follows.

\bigskip

\noindent{\bf Problem ($\textbf{FSLQ-P}_\lambda$)} \
For a given initial state $x\in\mathbb{R}^n $, find a control {$(u^*, v_1^*, v_2^*) \in \widetilde{\mathcal{U}}_{ad}$ such that
\[
\mathcal{J}_{\lambda}(x; u^*, v_1^*, v_2^*) = \inf_{(u, v_1, v_2) \in \widetilde{\mathcal{U}}_{ad}} \mathcal{J}_{\lambda}(x ; u, v_1, v_2) =: \mathcal{V}_{\lambda}(x)
\]

The following theorem reveals the connection between the forward and backward problems under partial information.

\begin{Theorem} \label{thm:bfr}
Let (H1)-(H3) hold. Then there exist constants $\alpha_0 > 0$ and $\lambda_0 > 0$, such that for 
any $\lambda \geqslant \lambda_0$,
\[
\mathcal{J}_{\lambda}(0;u,v_1,v_2) \geqslant \alpha_0\, \mathbb{E}\int_{0}^{T}\left(\left\lvert u(t)\right\rvert^2 + \left\lvert v_1(t)\right\rvert^2 + \left\lvert v_2(t)\right\rvert^2\right)\,\mathrm{d}t, \ \text{for all } (u,v_1,v_2) \in \widetilde{\mathcal{U}}_{ad}.
\]
Moreover, if $G = 0$, then for any $\lambda \geqslant \lambda_0$,
\[
\mathcal{J}_{\lambda}(x;u,v_1,v_2)\geqslant \alpha_0\, \mathbb{E}\int_{0}^{T}\left(\left\lvert u(t)\right\rvert^2 + \left\lvert v_1(t)\right\rvert^2 + \left\lvert v_2(t)\right\rvert^2\right)\, \mathrm{d}t,\ \text{for all } (u,v_1,v_2) \in \widetilde{\mathcal{U}}_{ad}, \text{ all } x \in \mathbb R^n.
\]
\end{Theorem}
\begin{proof}
Under (H1), for any given control $(u,v_1,v_2) \in \widetilde{\mathcal{U}}_{ad}$ and any initial state $x\in \mathbb{R}^n$, the state equation \eqref{FSLQn:SE} is uniquely solvable.
Let $X(\cdot)$ denote the solution to \eqref{FSLQn:SE} and set $\eta := X(T) \in L^2_{\mathcal{F}_T}(\Omega;\mathbb{R}^n)$.
We introduce the following BSDE:
\begin{equation}\label{BSDE:eta}
	\begin{cases}
		\begin{aligned}
		\mathrm{d}\widetilde{X}(t) &= 
		\big[ A(t)\widetilde{X}(t) + B(t)u(t) + C_{1}(t)\widetilde{Z}_{1}(t) + C_{2}(t)\widetilde{Z}_{2}(t) \big]\, \mathrm{d}t \\
		&+ \widetilde{Z}_{1}(t)\, \mathrm{d}W_{1}(t) + \widetilde{Z}_{2}(t)\, \mathrm{d}W_{2}(t), \\ 
		\widetilde{X}(T) &= \eta. 
		\end{aligned}
	\end{cases}
\end{equation}
With the uniqueness property, $(X, v_1, v_2)$ is the unique solution to \eqref{BSDE:eta}.
Let $(Y^{u,0}, Z_1^{u,0}, Z_2^{u,0})$ be the unique solution of
\begin{equation} \label{BSDE:Yu}
	\begin{cases}
	\begin{aligned}
	\mathrm{d}Y^{u,0}(t) &= 
	[ A(t)Y^{u,0}(t) + B(t)u(t) + C_1(t)Z^{u,0}_1(t) + C_2(t)Z^{u,0}_2(t) ]\, \mathrm{d}t \\
	&\quad+ Z^{u,0}_{1}(t)\, \mathrm{d}W_{1}(t) + Z^{u,0}_{2}(t)\, \mathrm{d}W_{2}(t), \\
	Y^{u,0}(T) &= 0 ,
	\end{aligned}
	\end{cases}
\end{equation}
and $(Y^{0,\eta}, Z_1^{0,\eta}, Z_2^{0,\eta})$ be the unique solution of
\begin{equation} \label{BSDE:Yx}
	\begin{cases}
	\begin{aligned}
	\mathrm{d}Y^{0,\eta}(t) &= 
	[ A(t)Y^{0,\eta}(t) + C_{1}(t)Z^{0,\eta}_{1}(t) + C_{2}(t)Z^{0,\eta}_{2}(t) ]\, \mathrm{d}t \\
	&\quad+ Z^{0,\eta}_{1}(t)\, \mathrm{d}W_{1}(t) + Z^{0,\eta}_{2}(t)\, \mathrm{d}W_{2}(t), \\
	Y^{0,\eta}(T) &= \eta.
	\end{aligned}
	\end{cases}
\end{equation}
By the linearity of the equations, we have
\[
X = Y^{u,0} + Y^{0,\eta}, \quad v_1 = Z^{u,0}_{1} + Z^{0,\eta}_{1}, \quad v_2 = Z^{u,0}_{2} + Z^{0,\eta}_{2}.
\]
Set 
\[
	M = \begin{pmatrix}
	Q &S_1^{\top} &S_2^{\top} & S_3^{\top}\\
	S_1& N_1& 0& 0 \\
	S_2& 0& N_2& 0 \\ 
	S_3& 0& 0& R
	\end{pmatrix}, \quad
	\gamma_1 = \begin{pmatrix}
	Y^{u,0} \\
	Z^{u,0}_1 \\
	Z^{u,0}_2 \\
	u
	\end{pmatrix}, \quad
	\gamma_2 = \begin{pmatrix}
	Y^{0,\eta} \\
	Z^{0,\eta}_1 \\
	Z^{0,\eta}_2 \\
	0
	\end{pmatrix}.
\]
In Problem (BSLQ-P), when the terminal state $\xi = 0$, the cost functional \eqref{BSLQ:CF} can then be rewritten as
\[
	J(0;u(\cdot)) = \mathbb{E} \left[\langle GY^{u,0}(0), \, Y^{u,0}(0)\rangle + \int_0^T\langle M(t) \gamma_1(t),\, \gamma_1(t)\rangle\, \mathrm{d}t\right],
\]
and thus, it follows from (H3) that
\begin{equation} \label{connection1}
	\begin{aligned}
	\mathcal{J}_\lambda(x;u, v_1, v_2) & = \mathbb{E}\left\{ \lambda|X(T)|^2+\int_0^T\langle M(t)[\gamma_1(t)+\gamma_2(t)],\, \gamma_1(t)+\gamma_2(t)\rangle\, \mathrm{d}t\right\} \\
	& = J(0;u(\cdot)) + \mathbb{E}\bigg[\lambda|X(T)|^2 - \langle GY^{u,0}(0),\, Y^{u,0}(0)\rangle \\
	&\quad + \int_0^T\langle M(t) \gamma_2(t),\, \gamma_2(t)\rangle\, \mathrm{d}t + 2\int_0^T\langle M(t) \gamma_1(t),\, \gamma_2(t)\rangle\, \mathrm{d}t\bigg] \\
	& \geqslant \delta\,\mathbb{E}\int_{0}^{T} |u(t)|^2\, \mathrm{d}t + \mathbb{E}\bigg[\lambda|X(T)|^2 - \langle GY^{u,0}(0),\, Y^{u,0}(0)\rangle \bigg] \\
	&\quad - \left\lvert \mathbb{E}\bigg[\int_0^T\langle M(t) \gamma_2(t),\, \gamma_2(t)\rangle\, \mathrm{d}t + 2\int_0^T\langle M(t) \gamma_1(t),\, \gamma_2(t)\rangle\, \mathrm{d}t\bigg] \right\rvert.
\end{aligned}
\end{equation}
Assumption (H2) indicates that the weighting matrices are all bounded. Hence, there exists a constant $K \geqslant 1$ such that
$|G| \leqslant K$ and $\left\lvert M(\cdot)\right\rvert \leqslant K $. Then we have
\begin{equation} \label{connection2}
	\begin{aligned}
	& \left|\mathbb{E}\left[\int_0^T\langle M(t) \gamma_2(t),\, \gamma_2(t)\rangle\, \mathrm{d}t+2 \int_0^T\langle M(t) \gamma_1(t),\, \gamma_2(t)\rangle\, \mathrm{d}t\right]\right| \\
	& \quad \leqslant K\left[\mathbb{E} \int_0^T|\gamma_2(t)|^2\, \mathrm{d}t + 2\,\mathbb{E} \int_0^T|\gamma_1(t)||\gamma_2(t)|\, \mathrm{d}t\right] \\
	& \quad \leqslant K\left[(\mu+1)\, \mathbb{E} \int_0^T|\gamma_2(t)|^2 \,\mathrm{d}t+\frac{1}{\mu}\, \mathbb{E} \int_0^T|\gamma_1(t)|^2 \,\mathrm{d}t\right],
	\end{aligned}
\end{equation}
where $\mu > 0$ is a constant to be chosen later. If we choose $K \geqslant 1$ large enough, then according to the standard estimate for BSDEs (see, e.g., Briand et al.~\cite[Proposition~2.2]{BCHMP00}), we have
\begin{equation} \label{connection3}
	\mathbb{E} \int_0^T|\gamma_1(t)|^2\, \mathrm{d}t \leqslant K\, \mathbb{E} \int_0^T|u(t)|^2\, \mathrm{d}t, \quad \mathbb{E} \int_0^T|\gamma_2(t)|^2\, \mathrm{d}t \leqslant K\,\mathbb{E}|\xi|^2 = K\,\mathbb{E}|X(T)|^2,
\end{equation}
and if the initial state $x = 0$, we further have 
\begin{equation} \label{connection5}
	|\langle GY^{u,0}(0), \, Y^{u,0}(0)\rangle| = |\langle GY^{0,\eta}(0), \, Y^{0,\eta}(0)\rangle| \leqslant K\, |Y^{0,\eta}(0)|^2 \leqslant K^2\,\mathbb{E}|X(T)|^2 .
\end{equation}
Using \eqref{connection3}, we obtain
\[
	\begin{aligned}
		\mathbb{E} \int_0^T\left[|v_1(t)|^2 + |v_2(t)|^2\right]\,\mathrm{d}t & = \mathbb{E} \int_0^T \left[|Z^{u,0}_1(t) + Z^{0,\eta}_1(t)|^2 + |Z^{u,0}_2(t) + Z^{0,\eta}_2(t)|^2 \right]\,\mathrm{d}t \\
		& \leqslant 2\,\mathbb{E}\int_0^T \left[\left|Z^{u,0}_1(t)\right|^2 + |Z^{u,0}_2(t)|^2\right] \mathrm{d}t \\
		&\quad+ 2\,\mathbb{E}\int_0^T \left[|Z^{0,\eta}_1(t)|^2 + |Z^{0,\eta}_2(t)|^2\right]\, \mathrm{d}t \\
		& \leqslant 2\,\mathbb{E} \int_0^T|\gamma_1(t)|^2\,\mathrm{d}t + 2\,\mathbb{E}\int_0^T|\gamma_2(t)|^2\,\mathrm{d}t \\
		& \leqslant 2K\,\mathbb{E}|X(T)|^2 + 2K\,\mathbb{E}\int_0^T |u(t)|^2\, \mathrm{d}t,
	\end{aligned}
\]
then
\begin{equation} \label{connection4}
	\mathbb{E}|X(T)|^2 \geqslant \frac{1}{2K}\,\mathbb{E} \int_0^T\left[|v_1(t)|^2 + |v_2(t)|^2\right]\,\mathrm{d}t - \mathbb{E}\int_0^T |u(t)|^2\, \mathrm{d}t.
\end{equation}
Combining \eqref{connection1}, \eqref{connection2}, \eqref{connection3}, we obtain
\[
	\begin{aligned}
	\mathcal{J}_\lambda(x;u, v_1, v_2) & \geqslant \left(\delta - \frac{K^2}{\mu}\right)\mathbb{E}\int_{0}^{T} |u(t)|^2\, \mathrm{d}t + \left(\lambda - K^2(\mu +2)\right)\mathbb{E}|X(T)|^2\\
	&\quad + K^2\,\mathbb{E}|X(T)|^2 - \langle GY^{u,0}(0),\, Y^{u,0}(0)\rangle.
	\end{aligned}
\]
Choosing $\mu = \frac{2K^2}{\delta}$ and $\lambda_0 = \frac{\delta}{4} + K^2(\mu + 2)$, due to \eqref{connection4}, for any $\lambda \geqslant \lambda_0$, we deduce that
\begin{equation}\label{connection6}
    \begin{aligned}
	\mathcal{J}_\lambda(x;u, v_1, v_2) &\geqslant \frac{\delta}{8K}\,\mathbb{E}\int_{0}^{T}\left[|u(t)|^2 + |v_1(t)|^2 + |v_2(t)|^2\right]\,\mathrm{d}t \\
	&\quad +\left[K^2\,\mathbb{E}|X(T)|^2 - \langle GY^{u,0}(0),\, Y^{u,0}(0) \rangle\right].
    \end{aligned}
\end{equation}
If the initial state $x = 0$, we obtain from \eqref{connection5} that
\[
	\mathcal{J}_\lambda(0; u,v_1,v_2) \geqslant \frac{\delta}{8K}\, \mathbb{E}\int_{0}^{T}\left[|u(t)|^2 + |v_1(t)|^2 + |v_2(t)|^2\right]\, \mathrm{d}t,\, \text{for all } (u,v_1,v_2) \in \widetilde{\mathcal{U}}_{ad},
\]
and if $G = 0$, \eqref{connection6} implies that
\[
	\mathcal{J}_\lambda(x; u,v_1,v_2) \geqslant \frac{\delta}{8K}\,\mathbb{E}\int_{0}^{T}\left[|u(t)|^2 + |v_1(t)|^2 + |v_2(t)|^2\right]\,\mathrm{d}t,\, \text{for all } (u,v_1,v_2) \in \widetilde{\mathcal{U}}_{ad}, \text{ all } x \in \mathbb R^n.
\]
\end{proof}
The above result generalizes Theorem 4.1 in Sun et al.~\cite{ILQ6}, 
which establishes the connection between forward and backward problems with complete information. We extend this connection to the framework with partial information.
Combining Theorem~\ref{thm:bfr} with Lemma~\ref{lemma:Req1} and Lemma~\ref{lemma:convexity}, we have the following two corollaries.
\begin{Corollary} \label{coro1}
	Let (H1) -- (H3) hold and $\lambda_0$ be the positive constant given in Theorem \ref{thm:bfr}. Then Problem ($\text{FSLQ-P}_\lambda$) is uniquely solvable for any $\lambda \geqslant \lambda_0$.
Moreover, if $G = 0$, then for any $\lambda \geqslant \lambda_0$, the value function $\mathcal{V}_\lambda$ satisfies
\[
\mathcal{V}_\lambda(x) \geqslant 0, \quad \text{for any } x \in \mathbb{R}^n.
\]
\end{Corollary}

\begin{Corollary} \label{coro2}
	Let (H1) -- (H3) hold and $\lambda_0$ be the positive constant given in Theorem \ref{thm:bfr}. Then, for any $\lambda \geqslant \lambda_0$, the Lyapunov equation
\begin{equation} \label{FSLQn:RE1}
	\begin{cases}
		\dot{\mathcal{P}_1}_\lambda + {\mathcal{P}_1}_\lambda A + A^{\top}{\mathcal{P}_1}_\lambda + Q  = 0,\\
	{\mathcal{P}_1}_\lambda(T) = \lambda I
	\end{cases}
\end{equation}
and Riccati equation
\begin{equation} \label{FSLQn:RE2}
		\begin{cases}
		\dot{\mathcal{P}_2}_\lambda + {\mathcal{P}_2}_\lambda A + A^{\top}{\mathcal{P}_2}_\lambda + Q \\
	  \quad- \begin{pmatrix}
			C_1^{\top}{\mathcal{P}_2}_\lambda + S_1 \\ C_2^{\top}{\mathcal{P}_2}_\lambda + S_2 \\ B^{\top}{\mathcal{P}_2}_\lambda + S_3
		\end{pmatrix}^{\top}
		\begin{pmatrix}
			N_1+{\mathcal{P}_1}_\lambda & 0 & 0\\
			0 & N_2+{\mathcal{P}_2}_\lambda & 0\\
			0 & 0 & R
		\end{pmatrix}^{-1}
		\begin{pmatrix}
			C_1^{\top}{\mathcal{P}_2}_\lambda + S_1 \\ C_2^{\top}{\mathcal{P}_2}_\lambda + S_2 \\ B^{\top}{\mathcal{P}_2}_\lambda + S_3
		\end{pmatrix} 
		= 0, \\
		{\mathcal{P}_2}_\lambda(T) = \lambda I
	\end{cases}
\end{equation}
admit unique solutions ${\mathcal{P}_1}_\lambda, {\mathcal{P}_2}_\lambda \in C(0,T;\mathbb{S}^n)$, respectively, such that
\begin{equation} \label{coro2: uniformly positive}
	\begin{pmatrix}
		N_1+{\mathcal{P}_1}_\lambda & 0 & 0\\
		0 & N_2+{\mathcal{P}_2}_\lambda & 0\\
		0 & 0 & R
	\end{pmatrix} \gg 0.
	\end{equation}
Consequently, in the cost functional \eqref{BSLQ:CF}, the control weighting matrix $R(\cdot) \gg 0$.
The value function is given by
\[
\mathcal{V}_\lambda(x) = \langle {\mathcal{P}_2}_\lambda(0)x, \, x\rangle.
\]
\end{Corollary}

\subsection{Simplification of the cost functional} \label{section:simplification}
To better use the previous results about Problem ($\text{FSLQ-P}_\lambda$) and simplify the subsequent calculations, we make some reductions for Problem (BSLQ-P) in this section, similar to \cite{ILQ6}. Specifically, consider the following linear ODE:
\[
	\begin{cases}
		\dot{\varPhi}(t) + \varPhi(t)A(t) + A(t)^{\top}\varPhi(t) + Q(t)  = 0,\\
		\varPhi(0) = -G.
	\end{cases}
\]
By applying It\^o's formula to $\langle \varPhi(\cdot)Y(\cdot), \ Y(\cdot) \rangle$, we obtain
\begin{equation} \label{GY0}
	\begin{aligned}
		\mathbb{E}\langle GY(0),\, Y(0)\rangle &= \mathbb{E}\int_{0}^{T}\mathrm{d} \big\langle \varPhi(t) Y(t),\, Y(t)\big\rangle -\mathbb{E}\big\langle \varPhi(T)\xi,\, \xi\big\rangle \\
		& =\mathbb{E} \int_0^T\Big[\big\langle(\dot{\varPhi}+\varPhi A+A^{\top} \varPhi) Y,\, Y\big\rangle + 2\big\langle B^{\top}\varPhi Y,\, u\big\rangle + 2\big\langle C_1^{\top}\varPhi Y,\, Z_1\big\rangle \\
		& \quad +2\big\langle C_2^{\top}\varPhi Y,\, Z_2\big\rangle+\big\langle \varPhi Z_1,\, Z_1\big\rangle + \big\langle \varPhi Z_2,\, Z_2\big\rangle\Big] \,\mathrm{d}t - \mathbb{E}\big\langle \varPhi(T)\xi,\, \xi\big\rangle\\
		& =\mathbb{E} \int_0^T\Big[-\big\langle QY,\, Y\big\rangle + \big\langle \varPhi Z_1,\, Z_1\big\rangle + \big\langle \varPhi Z_2,\, Z_2\big\rangle + 2\big\langle B^{\top}\varPhi Y,\, u\big\rangle\\
		& \quad+2\big\langle C_1^{\top}\varPhi Y,\, Z_1\big\rangle +2\big\langle C_2^{\top}\varPhi Y,\, Z_2\big\rangle\Big] \,\mathrm{d}t - \mathbb{E}\big\langle \varPhi(T)\xi,\, \xi\big\rangle.
	\end{aligned}
\end{equation}
Combining \eqref{GY0} and the transformations
\begin{equation}\label{Transformations}
\begin{aligned}
&N_1^\varPhi = N_1 + \varPhi,\quad N_2^\varPhi = N_2 + \varPhi,\\
&S_1^\varPhi = S_1 + C_1^{\top}\varPhi,\quad S_2^\varPhi = S_2 + C_2^{\top}\varPhi,\quad S_3^\varPhi = S_3+B^{\top}\varPhi, \\
\end{aligned}
\end{equation}
we obtain a new form of cost functional
\[
\begin{aligned}
	J(\xi; u(\cdot)) =& \ \mathbb{E}
		\int_0^T 
		\Big[
		\big\langle N_1^\varPhi Z_1,\, Z_1\big\rangle
		+\big\langle N_2^\varPhi Z_2,\, Z_2 \big\rangle 
		+2\big\langle S_1^\varPhi Y,\, Z_1 \big\rangle +2\big\langle S_2^\varPhi Y,\, Z_2 \big\rangle \\
		&+2\big\langle S_3^\varPhi Y,\, u \big\rangle	+\big\langle Ru,\, u \big\rangle 
		\Big]
		\,\mathrm{d}t
	- \mathbb{E}\langle \varPhi(T)\xi,\xi\rangle.
\end{aligned}
\]
Note that $\varPhi(\cdot)$ is independent of control $u(\cdot)$ and the terminal state $\xi$ is given. Thus, our problem is equivalent to
minimizing
\[
\begin{aligned}
	J^\varPhi(\xi; u(\cdot)) =& \ \mathbb{E}
    	\int_0^T 
	    \Big[
	    \big\langle N_1^\varPhi Z_1,\, Z_1\big\rangle
	    +\big\langle N_2^\varPhi Z_2,\, Z_2 \big\rangle 
	    +2\big\langle S_1^\varPhi Y,\, Z_1 \big\rangle +2\big\langle S_2^\varPhi Y,\, Z_2 \big\rangle \\
    	&+2\big\langle S_3^\varPhi Y,\, u \big\rangle	+\big\langle Ru,\, u \big\rangle 
    	\Big]
	    \,\mathrm{d}t
\end{aligned}
\]
over $\mathcal{U}_{ad}$, subject to the state equation~\eqref{BSLQ:SE}.
For this reason, without loss of generality, we may assume the following condition in the rest of the paper:
\begin{equation} \label{simplification}
	G=0, \quad Q(\cdot)=0.
\end{equation}

One direct result from this simplification is the following proposition.

\begin{Proposition} \label{prop:P_lambda increasing}
	Suppose (H1) -- (H3) and \eqref{simplification} hold, and let $\lambda_0$ be the positive constant given in Theorem \ref{thm:bfr}. For any $\lambda \geqslant \lambda_0$, let ${\mathcal{P}_1}_\lambda(\cdot)$ and ${\mathcal{P}_2}_\lambda(\cdot)$ be 
	the solutions to Lyapunov equation~\eqref{FSLQn:RE1} and Riccati equation~\eqref{FSLQn:RE2}, respectively. Then we have
	 \[
	 {\mathcal{P}_1}_\lambda(\cdot) > 0, \quad {\mathcal{P}_2}_\lambda(\cdot) \geqslant 0,
	 \]
	 and for any $\lambda_2 > \lambda_1 \geqslant \lambda_0$, we have
\[
	 {\mathcal{P}_1}_{\lambda_2}(\cdot)-{\mathcal{P}_1}_{\lambda_1}(\cdot)>0, \quad {\mathcal{P}_2}_{\lambda_2}(\cdot) - {\mathcal{P}_2}_{\lambda_1}(\cdot)>0.
\]
\end{Proposition}
\begin{proof}
We prove the property of ${\mathcal{P}_1}_\lambda(\cdot)$ first. Let $\mathcal{P}^1(\cdot) = {\mathcal{P}_1}_{\lambda_2}(\cdot) - {\mathcal{P}_1}_{\lambda_1}(\cdot)$, satisfying
\[
	\left\{\begin{array}{l}
		\dot{\mathcal{P}}^1 + {\mathcal{P}^1}A + A^{\top}{\mathcal{P}^1} = 0,\\
	{\mathcal{P}^1}(T) = (\lambda_2 - \lambda_1) I.
	\end{array}\right.
\]
From Lemma~\ref{lemma:Req1} it follows immediately that ${\mathcal{P}_1}_\lambda(\cdot) > 0$ and ${\mathcal{P}^1}(\cdot) > 0$,
which means ${\mathcal{P}_1}_{\lambda_2}(\cdot) - {\mathcal{P}_1}_{\lambda_1}(\cdot)>0$.

Now, as for Riccati equation~\eqref{FSLQn:RE2}, since $G = 0$, from Corollary~\ref{coro1} and Corollary~\ref{coro2} we have 
\[
 \langle {\mathcal{P}_2}_\lambda(0)x, x\rangle=\mathcal{V}_\lambda(x)  \geqslant 0,\quad \text{for any } x\in\mathbb{R}^n,
\]
which indicates ${\mathcal{P}_2}_\lambda(0) \geqslant 0$. Let $\varPi(\cdot)$ be the solution to the following linear ODE
\[
	\begin{cases}
	\dot{\varPi}(t) = A(t)\varPi(t),\\
	\varPi(0) = I.
	\end{cases}
\]
By the integration by parts formula, we obtain
\[
\begin{split}
\varPi(t)^{\top}{\mathcal{P}_2}_\lambda(t)\varPi(t) &= {\mathcal{P}_2}_\lambda(0) + \int_{0}^{t}\mathrm{d}\, \big[\varPi(s)^{\top}{\mathcal{P}_2}_\lambda(s)\varPi(s)\big] \\
& = {\mathcal{P}_2}_\lambda(0) + \int_{0}^{t}\varPi(s)^{\top}\big[\dot{\mathcal{P}_2}_\lambda(s) + {\mathcal{P}_2}_\lambda(s)A(s) + A(s)^{\top}{\mathcal{P}_2}_\lambda(s)\big]\varPi(s) \,\mathrm{d}s \\
& = {\mathcal{P}_2}_\lambda(0) + \int_{0}^{t}\varPi(s)^{\top}\mathcal{Q}_\lambda(s)\varPi(s) \,\mathrm{d}s, 
\end{split}
\]
where
\[
\mathcal{Q}_\lambda = \begin{pmatrix}
	C_1^{\top}{\mathcal{P}_2}_\lambda + S_1 \\ C_2^{\top}{\mathcal{P}_2}_\lambda + S_2 \\ B^{\top}{\mathcal{P}_2}_\lambda + S_3
\end{pmatrix}^{\top}
\begin{pmatrix}
	N_1+{\mathcal{P}_1}_\lambda & 0 & 0\\
	0 & N_2+{\mathcal{P}_2}_\lambda & 0\\
	0 & 0 & R
\end{pmatrix}^{-1}
\begin{pmatrix}
	C_1^{\top}{\mathcal{P}_2}_\lambda + S_1 \\ C_2^{\top}{\mathcal{P}_2}_\lambda + S_2 \\ B^{\top}{\mathcal{P}_2}_\lambda + S_3
\end{pmatrix}. 
\]
Thanks to the invertibility of $\varPi(\cdot)$, we get
\[
	{\mathcal{P}_2}_\lambda(t) = \left[\varPi^{-1}(t)\right]^{\top}\left[ {\mathcal{P}_2}_\lambda(0) + \int_{0}^{t}\varPi^{\top}(s)\mathcal{Q}_\lambda(s)\varPi(s) \,\mathrm{d}s \right]\varPi^{-1}(t), \quad \text{for any } t\in[0,T].
\]
By Corollary~\ref{coro2}, $\mathcal{Q}_\lambda(\cdot) \geqslant 0$, and together with ${\mathcal{P}_2}_\lambda(0) \geqslant 0$, it follows that ${\mathcal{P}_2}_\lambda(\cdot) \geqslant 0$.
As for the monotonicity of ${\mathcal{P}_2}_\lambda(\cdot)$ with respect to $\lambda$, just repeat the same procedure like ${\mathcal{P}_1}_\lambda(\cdot)$ and use Lemma~\ref{lemma:std eq2}. This completes the proof.
\end{proof}

\subsection{Construction of the optimal control} \label{section:CONSTRUCT of the optimal control}
\subsubsection{The Hamiltonian system and the matrix-valued  differential equation }
\begin{Theorem} \label{theorem: FBSDE}
	Let (H1) -- (H2) hold and $\xi\in L^2_{\mathcal{F}_T}(\Omega;\mathbb{R}^n)$ be given.
A control $u^*(\cdot)\in\mathcal{U}_{ad}$ is optimal if and only if the following conditions are satisfied:

(i)\, $J(0;u(\cdot)) \geqslant 0$ for all $u(\cdot)\in\mathcal{U}_{ad}$;

(ii)\, The solution $(X^*(\cdot), Y^*(\cdot), Z_1^*(\cdot), Z_2^*(\cdot))$ to the FBSDE
\begin{equation} \label{FBSDE:state equation}
	\begin{cases}
	\begin{aligned}
	&\mathrm{d}X^* = \big[QY^*-A^{\top} X^*+S_1^{\top}Z_1^*+S_2^{\top}Z_2^*+S_3^{\top}u^*\big] \,\mathrm{d}t \\
	&\quad\quad\quad+\big[-C_1^{\top} X^*+S_1Y^*+N_{1}Z_1^*\big] \,\mathrm{d}W_1 \\
	&\quad\quad\quad+\big[-C_2^{\top} X^*+S_2Y^*+N_{2}Z_2^*\big] \,\mathrm{d}W_2, \\	
	&\mathrm{d}Y^*= \big[AY^*+B u^*+C_1 Z_1^*+C_2 Z_2^*\big] \,\mathrm{d}t+Z_1^* \,\mathrm{d}W_1 + Z_2^* \,\mathrm{d}W_2,\\
	&X^*(0)= GY^*(0), \quad Y^*(T)=\xi
	\end{aligned}
	\end{cases}
\end{equation}
satisfies
\begin{equation} \label{FBSDE: u* equation}
S_3\widehat{Y}^* - B^{\top}\widehat{X}^* + Ru^* = 0.
\end{equation}
\end{Theorem}

\begin{proof}
Clearly, $u^*(\cdot)$ is an optimal control if and only if for any $\varepsilon \in \mathbb{R}$ and any $u(\cdot)\in\mathcal{U}_{ad}$,
\begin{equation} \label{variation: optimal cond}
	J\left(\xi ; u^*(\cdot)+\varepsilon u(\cdot)\right)-J\left(\xi ; u^*(\cdot)\right) \geqslant 0.
\end{equation}
Denote $(Y^{\varepsilon}, Z^{\varepsilon}_1, Z^{\varepsilon}_2)$ as the solution to the following BSDE
\[
\begin{cases}\begin{aligned}
	\mathrm{d}Y^{\varepsilon}(t) & =\big\{A(t)Y^{\varepsilon}(t)+B(t)[u^*(t) + \varepsilon u(t)]+C_1(t) Z^{\varepsilon}_1(t) + C_2(t)Z^{\varepsilon}_2(t)\big\} \,\mathrm{d}t\\ 
	&\quad +Z_1(t) \,\mathrm{d}W_1(t) + Z_2(t) \,\mathrm{d}W_2(t), \\
	Y^{\varepsilon}(T) & = \xi,
\end{aligned}\end{cases}
\]
then
\[
    Y^{\varepsilon}= Y^*+\varepsilon  Y^{u,0}, \quad Z^{\varepsilon}_1= Z_1^* + \varepsilon Z_1^{u,0}, \quad Z^{\varepsilon}_2= Z_2^* + \varepsilon Z_2^{u,0}, 
\]
where $(Y^{u,0}, Z^{u,0}_1, Z^{u,0}_2)$ is the solution to BSDE~\eqref{BSDE:Yu}. Therefore,
\[
	\begin{aligned}
		& J\left(\xi ; u^*(\cdot)+\varepsilon u(\cdot)\right)-J\left(\xi ; u^*(\cdot)\right) \\
		=& \ 2 \varepsilon\, \mathbb{E}\bigg\{\big\langle G Y^*(0),\, Y^{u,0}(0)\big\rangle+\int_0^T
		\Big[
		\big\langle QY^* + S_1^{\top} Z_1^* + S_2^{\top} Z_2^* + S_3^{\top} u^*,\, Y^{u,0} \big\rangle \\
		&+\big\langle N_1 Z_1^* + S_1Y^*,\, Z_1^{u,0}\big\rangle
		+\big\langle N_2 Z_2^* + S_2Y^*,\, Z_2^{u,0} \big\rangle 
		+\big\langle S_3Y^* + Ru^*,\, u \big\rangle
		\Big]
		\,\mathrm{d}t
		\bigg\} + \varepsilon ^2 J(0;u(\cdot)).
	\end{aligned}
\]
By applying It\^o's formula to $\big\langle X^*(\cdot), \, Y^{u,0}(\cdot) \big\rangle$, we have
\[
\begin{aligned}
\mathbb{E}\big\langle G Y^*(0),\, Y^{u,0}(0)\big\rangle &= -\mathbb{E} \int_0^T\Big[\big\langle QY^* + S_1^{\top} Z_1^*+S_2^{\top} Z_2^*+S_3^{\top} u^*,\, Y^{u,0}\big\rangle +\left\langle S_1 Y^*+N_1 Z_1^*,\, Z_1^{u,0}\right\rangle \\
&\quad\quad+\left\langle S_2 Y^*+N_2 Z_2^*,\, Z_2^{u,0}\right\rangle+\big\langle B^{\top} X^*,\, u\big\rangle\Big] \,\mathrm{d}t.
\end{aligned}
\]
We combine the above equations and use properties of conditional expectation to obtain
\[
\begin{aligned}
J\left(\xi ; u^*(\cdot)+\varepsilon u(\cdot)\right)-J\left(\xi ; u^*(\cdot)\right) &= \varepsilon ^2J(0;u(\cdot)) + 2\varepsilon\,\mathbb{E}\int_{0}^{T}\big\langle S_3Y^* - B^{\top}X^* + Ru^*,\, u\big\rangle \,\mathrm{d}t \\
&= \varepsilon ^2J(0;u(\cdot)) + 2\varepsilon\,\mathbb{E}\int_{0}^{T}\big\langle S_3\widehat{Y}^* - B^{\top}\widehat{X}^* + Ru^*,\, u\big\rangle \,\mathrm{d}t.
\end{aligned}
\]
Due to the arbitrariness of $\varepsilon$ and $u(\cdot)$, \eqref{variation: optimal cond} holds if and only if (i) and (ii) are satisfied.
\end{proof}
Under the simplification condition \eqref{simplification}, FBSDE~\eqref{FBSDE:state equation} together with the constraint \eqref{FBSDE: u* equation} becomes the following Hamiltonian system:
\begin{equation} \label{FBSDE:Hamilton System}
	\begin{cases}\begin{aligned}
	&\mathrm{d}X= \big[-A^{\top} X+S_1^{\top}Z_1+S_2^{\top}Z_2+S_3^{\top}u\big] \,\mathrm{d}t \\
	&\quad\quad\,\, +\big[-C_1^{\top} X + S_1Y + N_1Z_1\big] \,\mathrm{d}W_1 \\
	&\quad\quad\,\, +\big[-C_2^{\top} X + S_2Y + N_2 Z_2\big]\, \mathrm{d}W_2, \\	
	&\mathrm{d}Y= \big[AY+B u+C_1 Z_1+C_2 Z_2\big]\, \mathrm{d}t+Z_1\,\mathrm{d}W_1+Z_2\,\mathrm{d}W_2, \\
	&X(0)= 0, \quad Y(T)=\xi, \\
	&S_3\widehat{Y} - B^{\top}\widehat{X} + Ru = 0.
\end{aligned}\end{cases}
\end{equation}
When (H3) holds, it follows from Corollary~\ref{coro2} that $R(\cdot) \gg 0$, which implies that $R(\cdot)$ is invertible.
Thus, we obtain from the last equation in \eqref{FBSDE:Hamilton System}
\begin{equation} \label{DECOUPLING:u=R}
	u(\cdot) = -R(\cdot)^{-1}\big[S_3(\cdot)\widehat{Y}(\cdot) - B(\cdot)^{\top}\widehat{X}(\cdot)\big].
\end{equation}
Substituting it into \eqref{FBSDE:Hamilton System} leads to
\begin{equation} \label{FBSDE-F}
	\begin{cases}\begin{aligned}
			&\mathrm{d}X= \big[-A^{\top} X+S_1^{\top}Z_1+S_2^{\top}Z_2+S_3^{\top}u\big] \,\mathrm{d}t \\
			&\quad\quad\,\, +\big[-C_1^{\top} X + S_1Y + N_1Z_1\big] \,\mathrm{d}W_1 \\
			&\quad\quad\,\, +\big[-C_2^{\top} X + S_2Y + N_2 Z_2\big]\, \mathrm{d}W_2, \\	
			&\mathrm{d}Y= \big[AY-BR^{-1}S_3\widehat{Y}+BR^{-1}B^{\top}\widehat{X}+C_1 Z_1+C_2 Z_2\big]\, \mathrm{d}t\\
			&\quad\quad\,\, +Z_1\,\mathrm{d}W_1+Z_2\,\mathrm{d}W_2, \\
			&X(0)= 0, \quad Y(T)=\xi, \\
	\end{aligned}\end{cases}
\end{equation}
which is actually coupled and incorporates filtering. In order to decouple the FBSDE with filtering, similarly to Wang et al.~\cite{PLQ8}, we assume that
\begin{equation} \label{DECOUPLING: Y=GX}
Y(\cdot) = -\varGamma(\cdot)\widehat{X}(\cdot) + \varphi(\cdot),
\end{equation}
where $\varGamma(\cdot)$ is a deterministic matrix-valued function and $\varphi(\cdot)$ is a stochastic process that satisfies the following BSDE
\[
\left\{\begin{aligned}
	\mathrm{d}\varphi(t) &= \alpha(t) \,\mathrm{d}t + {\beta}_1(t) \,\mathrm{d}W_1 + {\beta}_2(t) \,\mathrm{d}W_2,\\
	\varphi(T) & = \xi
\end{aligned}\right.
\]
where $\alpha(\cdot)$ will be determined below.
 
Applying Itô's formula to \eqref{DECOUPLING: Y=GX}, we have
\[
	\begin{aligned} 
		0= & -\mathrm{d}Y-\dot{\varGamma} \widehat{X} \,\mathrm{d}t-\varGamma \,\mathrm{d}\widehat{X} + \mathrm{d}\varphi \\
		= & -(AY + Bu + C_1 Z_1 + C_2 Z_2) \,\mathrm{d}t -Z_1 \,\mathrm{d}W_1 - Z_2 \,\mathrm{d}W_2- \dot{\varGamma} \widehat{X} \,\mathrm{d}t\\
		& - \varGamma(-A^{\top} \widehat{X}+S_1^{\top}\widehat{Z}_1+S_2^{\top}\widehat{Z}_2+S_3^{\top}u)\,\mathrm{d}t \\
		& - \varGamma(-C_2^{\top} \widehat{X}+S_2\widehat{Y}+N_2\widehat{Z}_2)\,\mathrm{d}W_2 + \alpha \,\mathrm{d}t + {\beta}_1 \,\mathrm{d}W_1+ {\beta}_2 \,\mathrm{d}W_2,
	\end{aligned} 
\]
from which we obtain
\begin{equation} \label{DECOUPLING:each term be zero}
	\begin{cases} 
	\dot{\varGamma}\widehat{X} + AY - \varGamma A^{\top}\widehat{X} + (B+\varGamma S^{\top}_3)u + C_1 Z_1 + \varGamma S_1^{\top}\widehat{Z}_1+ C_2 Z_2 + \varGamma S_2^{\top} \widehat{Z}_2 - \alpha= 0, \\
	Z_1 - {\beta}_1=0, \\
	Z_2 + \varGamma (-C_2^{\top} \widehat{X}+S_2\widehat{Y}+N_2\widehat{Z}_2) - {\beta}_2=0.
	\end{cases}
\end{equation}
For convenience, we adopt the following notations:
\[
\begin{aligned}
	&\mathcal{N}_{\varGamma}(\cdot) = I + \varGamma(\cdot) N_2(\cdot), \\
	&\mathcal{B}_{\varGamma}(\cdot) = B(\cdot) + \varGamma(\cdot) S_3(\cdot)^{\top},\\
	&\mathcal{C}_{\varGamma}(\cdot) = C_2(\cdot) + \varGamma(\cdot) S_2(\cdot)^{\top}.
\end{aligned}
\]
Here, we temporarily assume that $\mathcal{N}_{\varGamma}(\cdot)$ is invertible (its invertibility will be established in Theorem \ref{thm:bRiccati} below). We further obatin
\begin{equation} \label{DECOUPLING:Z1Z2}
	\begin{cases} 
	Z_1  = {\beta}_1, \\
	Z_2  = \mathcal{N}_{\varGamma}^{-1}\left(\varGamma \mathcal{C}_{\varGamma}^{\top} \widehat{X} - \varGamma S_2\widehat{\varphi} + {\widehat{\beta}}_2\right) + {\beta}_2 - {\widehat{\beta}}_2.
	\end{cases}
\end{equation}
By substituting \eqref{DECOUPLING:u=R}, \eqref{DECOUPLING: Y=GX}, \eqref{DECOUPLING:Z1Z2} into the first equation of \eqref{DECOUPLING:each term be zero}, we finally obtain
\[
\begin{aligned}
&(
\dot{\varGamma} -A \varGamma-\varGamma A^{\top} + \mathcal{B}_{\varGamma}R^{-1} \mathcal{B}_{\varGamma}^{\top}
	+\mathcal{C}_{\varGamma}\mathcal{N}_{\varGamma}^{-1} \varGamma \mathcal{C}_{\varGamma}^{\top} 
)\widehat{X} -\alpha + A\varphi + C_1{\beta}_1 + C_2{\beta}_2 \\
&\quad - (\mathcal{B}_{\varGamma}R^{-1}S_3 + \mathcal{C}_{\varGamma}\mathcal{N}_{\varGamma}^{-1}\varGamma S_2)\widehat{\varphi} 
+\varGamma S_1^{\top} \widehat{\beta}_1  + (\mathcal{C}_{\varGamma}\mathcal{N}_{\varGamma}^{-1} - C_2 )\widehat{\beta}_2 = 0,
\end{aligned}
\]
which yields the following matrix-valued differential equation
\begin{equation} \label{DECOUPLING:Riccati Equation}
	\begin{cases}
	\dot{\varGamma} -A \varGamma-\varGamma A^{\top}+\mathcal{B}_{\varGamma}R^{-1} \mathcal{B}_{\varGamma}^{\top}+\mathcal{C}_{\varGamma}{\mathcal{N}}_{\varGamma}^{-1} \varGamma \mathcal{C}_{\varGamma}^{\top}=0,\\
	\varGamma(T) = 0,
	\end{cases}
\end{equation}
and $\varphi(\cdot)$ is the solution to the following BSDE with filtering
\begin{equation} \label{DECOUPLING:adjoint BSDE}
	\begin{cases}
	\mathrm{d}\varphi = \big[A\varphi + C_1{\beta}_1 + C_2{\beta}_2
	- (\mathcal{B}_{\varGamma}R^{-1}S_3 + \mathcal{C}_{\varGamma}\mathcal{N}_{\varGamma}^{-1}\varGamma S_2)\widehat{\varphi} \\
	\quad\quad\,\,\,+ \varGamma S_1^{\top} \widehat{\beta}_1  + (\mathcal{C}_{\varGamma}\mathcal{N}_{\varGamma}^{-1} - C_2 )\widehat{\beta}_2\big] \,\mathrm{d}t + {\beta}_1 \,\mathrm{d}W_1 + {\beta}_2 \,\mathrm{d}W_2,\\
	\varphi(T) = \xi.
	\end{cases}
\end{equation}

Comparing the coefficients between the state equation~\eqref{BSLQ:SE} and the matrix-valued differential equation \eqref{DECOUPLING:Riccati Equation},
we find that $C_1(\cdot)$ is not involved in the equation. Moreover, it can be seen from \eqref{DECOUPLING:Z1Z2} that process $Z_1(\cdot)$
is completely determined by $\beta_1(\cdot)$. These happen because we can only observe partial information $\mathbb{G}$ rather than complete information $\mathbb{F}$.
The following theorem establishes the solvability of equation~\eqref{DECOUPLING:Riccati Equation} and BSDE~\eqref{DECOUPLING:adjoint BSDE}.

\begin{Theorem} \label{thm:bRiccati}
	Let (H1) -- (H3) and \eqref{simplification} hold. Then, there exists $\varGamma(\cdot)\in C([0,T];\mathbb{S}_{+}^n)$ 
	such that $\mathcal{N}_{\varGamma}(\cdot)$ is invertible and $\mathcal{N}_{\varGamma}(\cdot)^{-1}\in L^{\infty}(0,T;\mathbb{R}^{n \times n})$. 
	Moreover,  $\varGamma(\cdot)$ is the unique solution to equation \eqref{DECOUPLING:Riccati Equation}.
	Consequently, BSDE with filtering~\eqref{DECOUPLING:adjoint BSDE} admits a unique solution $(\varphi(\cdot), \beta_1(\cdot), \beta_2(\cdot)) \in S_{\mathbb{F}}^2(0, T; \mathbb{R}^n) \times L_{\mathbb{F}}^2(0, T; \mathbb{R}^n) \times L_{\mathbb{F}}^2(0, T; \mathbb{R}^n)$.
\end{Theorem}
\begin{proof}
	Proposition~\ref{prop:P_lambda increasing} shows that for $\lambda > {\lambda}_0$, ${\mathcal{P}_1}_{\lambda}(\cdot)$ and ${\mathcal{P}_2}_{\lambda}(\cdot)$ are positive definite and increasing in $\lambda$. We define
\[
	\varSigma_{\lambda}(\cdot) = \mathcal{P}_{1\lambda}(\cdot)^{-1}, \quad \varGamma_{\lambda}(\cdot) = \mathcal{P}_{2\lambda}(\cdot)^{-1},
\]
both of which are decreasing in $\lambda$ and bounded below by 0. Therefore, the family $\{\varSigma_{\lambda}(\cdot)\}_{\lambda > \lambda_{0}}$ and 
$\{\varGamma_{\lambda}(\cdot)\}_{\lambda > \lambda_{0}}$ are uniformly bounded and converge pointwise to some positive semi-definite functions $\varSigma(\cdot)$ and $\varGamma(\cdot)$, respectively. We will prove in three steps that the $\varGamma(\cdot)$ is the desired solution.

\medskip

{\bf Step 1: $\mathcal{N}_{\varGamma}(\cdot):=I_n+\varGamma(\cdot)N_2(\cdot)$ is invertible and $\mathcal{N}_{\varGamma}(\cdot)^{-1}\in L^{\infty}(0,T;\mathbb{R}^{n \times n})$.}

\medskip

We first investigate the properties of $\varSigma(\cdot)$. 
Note that $\varSigma_{\lambda}(\cdot)\mathcal{P}_{1\lambda}(\cdot) = I_n$.
With the identify
\[
	\dot{\varSigma}_{\lambda}(t) \mathcal{P}_{1\lambda}(t) + {\varSigma}_{\lambda}(t) \dot{\mathcal{P}}_{1\lambda}(t) = \frac{\mathrm{d}}{\mathrm{d}t}({\varSigma}_{\lambda}(t) \mathcal{P}_{1\lambda}(t)) = 0,
\]
we have
\[
\begin{aligned}
	\varSigma(t) &= \lim_{\lambda \rightarrow \infty}{\varSigma}_{\lambda}(t) \\
	&= \lim_{\lambda \rightarrow \infty}\left\{\frac{1}{\lambda} I_n - \int_{0}^{t}\big[{\varSigma}_{\lambda}(s)\dot{\mathcal{P}}_{1\lambda}(s){\varSigma}_{\lambda}(s)\big]\, \mathrm{d}s\right\} \\
	&= \lim_{\lambda \rightarrow \infty}\left\{\frac{1}{\lambda} I_n + \int_{0}^{t}\big[{\varSigma}_{\lambda}(s)({\mathcal{P}_1}_\lambda(s) A(s) + A(s)^{\top}{\mathcal{P}_1}_\lambda(s)){\varSigma}_{\lambda}(s)\big]\, \mathrm{d}s\right\} \\
	&= \lim_{\lambda \rightarrow \infty}\left\{\frac{1}{\lambda} I_n + \int_{0}^{t}\big[A(s){\varSigma}_{\lambda}(s) + {\varSigma}_{\lambda}(s)A(s)^{\top}\big]\, \mathrm{d}s\right\}\\
	&= \int_{0}^{t}\big[A(s){\varSigma}(s) + {\varSigma}(s)A(s)^{\top}\big]\, \mathrm{d}s,
\end{aligned}
\]
where the last equality is guaranteed by the dominated convergence theorem.
Thus $\varSigma(\cdot) \in C([0,T];\mathbb{S}_{+}^n)$ and satisfies
\[
	\left\{\begin{aligned}
		&\dot{\varSigma} -A \varSigma-\varSigma A^{\top} = 0, \\
		&\varSigma(T) = 0,
	\end{aligned}\right.
\]
which implies that $\varSigma(\cdot) = 0$.

For convenience, let
\[
	N = \begin{pmatrix} N_1 & 0\\ 0 & N_2 \end{pmatrix},\quad \mathcal{P}_\lambda = \begin{pmatrix}\mathcal{P}_{1\lambda} & 0\\ 0 & \mathcal{P}_{2\lambda}\end{pmatrix},\quad \varXi_\lambda = \begin{pmatrix}\varSigma_{\lambda} & 0\\ 0 & \varGamma_{\lambda}\end{pmatrix}.
\]
By Corollary~\ref{coro2} and Proposition~\ref{prop:P_lambda increasing}, for each $\lambda > \lambda_0$, we get
\[
	\mathcal{P}_\lambda (\varXi_\lambda N + I_{2n})=N + \mathcal{P}_\lambda \gg 0,
\]
and 
\[
N + \mathcal{P}_\lambda > N + \mathcal{P}_{\lambda_0}.
\]
Then $\varXi_\lambda N + I_{2n}$ is invertible and 
$|(N + \mathcal{P}_\lambda)^{-1}| < |(N + \mathcal{P}_{\lambda_0})^{-1}|$. Hence, for 
any $x\in\mathbb{R}^n$,
\[
\begin{aligned}
\left\lvert (I_{2n} + \varXi_\lambda N)^{-1} x \right\rvert^2 &= \left\lvert (\mathcal{P}_\lambda + N)^{-1} \mathcal{P}_\lambda x\right\rvert^2\\
&\leqslant 2\left\lvert (\mathcal{P}_\lambda + N)^{-1} (\mathcal{P}_\lambda - \mathcal{P}_{\lambda_0}) x\right\rvert^2 + 2\left\lvert (\mathcal{P}_\lambda + N)^{-1} \mathcal{P}_{\lambda_0} x\right\rvert^2\\
&= 2\left\lvert x -(\mathcal{P}_\lambda + N)^{-1} (\mathcal{P}_{\lambda_0} + N) x\right\rvert^2 + 2\left\lvert (\mathcal{P}_\lambda + N)^{-1} \mathcal{P}_{\lambda_0} x\right\rvert^2\\
&\leqslant 4\left[1 + \left\lvert (\mathcal{P}_{\lambda_0} + N)^{-1}\right\lvert^2\left(\left\lvert \mathcal{P}_{\lambda_0} + N\right\lvert^2 + \left\lvert \mathcal{P}_{\lambda_0} \right\lvert^2\right)\right]\left\lvert x \right\lvert^2\\
&\leqslant K\left\lvert x \right\lvert^2,
\end{aligned}
\]
where $K>0$ is a constant independent of $\lambda$. Therefore,
\[
	\left\lvert I_{2n} + \varXi_\lambda N \right\lvert \geqslant \frac{1}{\left\lvert(I_{2n} + \varXi_\lambda N)^{-1}\right\lvert}
	\geqslant \frac{1}{\sqrt{K}} =: \delta_0,
\]
which further implies that for each $\lambda > \lambda_{0}$,
\begin{equation} \label{inequality for R_SIG}
	(I_{2n} + \varXi_\lambda N)(I_{2n} + \varXi_\lambda N)^{\top} \geqslant \delta_0^2 I_{2n}.
\end{equation}
Since
\[
\begin{aligned}
\lim_{\lambda \rightarrow \infty} (I_{2n} + \varXi_\lambda N) &= \lim_{\lambda \rightarrow \infty} \left[I_{2n} + \begin{pmatrix}\varSigma_{\lambda} & 0\\ 0 & \varGamma_{\lambda}\end{pmatrix}\begin{pmatrix} N_1 & 0\\ 0 & N_2 \end{pmatrix}\right]
= \begin{pmatrix} I_n & 0\\ 0 & \mathcal{N}_{\varGamma} \end{pmatrix},
\end{aligned}
\]
letting $\lambda \rightarrow \infty$ in \eqref{inequality for R_SIG}, we obtain
\[
  \mathcal{N}_{\varGamma}\mathcal{N}_{\varGamma}^{\top} \geqslant \delta_0^2I_{n}.
\]
This leads to the conclusion that $\mathcal{N}_{\varGamma}(\cdot)$ is invertible and $\mathcal{N}_{\varGamma}(\cdot)^{-1} \in L^{\infty}(0,T;\mathbb{R}^{n \times n})$.

\medskip

{\bf Step 2: $\varGamma(\cdot) \in C([0,T];\mathbb{S}_{+}^n)$ is a solution to equation \eqref{DECOUPLING:Riccati Equation}.}

\medskip

We have known that $\varGamma(\cdot)$ is positive semi-definite, it remains to show that $\varGamma(\cdot)$ is continuous and satisfies the equation \eqref{DECOUPLING:Riccati Equation}. 
Note that $\varGamma_{\lambda}(\cdot)\mathcal{P}_{2\lambda}(\cdot) = I_n$.
Using the identity
\[
	\dot{\varGamma}_{\lambda}(t) \mathcal{P}_{2\lambda}(t) + {\varGamma}_{\lambda}(t) \dot{\mathcal{P}}_{2\lambda}(t) = \frac{\mathrm{d}}{\mathrm{d}t}({\varGamma}_{\lambda}(t) \mathcal{P}_{2\lambda}(t)) = 0,
\]
it follows that
\[
\begin{aligned}
	\dot{{\varGamma}}_{\lambda} &= - {\varGamma}_{\lambda}\dot{\mathcal{P}}_{2\lambda}{\varGamma}_{\lambda} \\
&=  A{\varGamma}_{\lambda} + {\varGamma}_{\lambda}A^{\top} - 
  \begin{pmatrix}
	C_1^{\top} + S_1{\varGamma}_{\lambda}\\ C_2^{\top} + S_2{\varGamma}_{\lambda} \\ B^{\top} + S_3{\varGamma}_{\lambda}
	\end{pmatrix}^{\top}
  \begin{pmatrix}
	  N_1 + {\mathcal{P}_1}_\lambda & 0 & 0\\
	  0 & N_2+{\mathcal{P}_2}_\lambda & 0\\
	  0 & 0 & R
  \end{pmatrix}^{-1}
  \begin{pmatrix}
	  C_1^{\top} + S_1{\varGamma}_{\lambda}\\ C_2^{\top} + S_2{\varGamma}_{\lambda} \\ B^{\top} + S_3{\varGamma}_{\lambda}
  \end{pmatrix}\\
&= A {\varGamma}_{\lambda} + {\varGamma}_{\lambda} A^{\top} - 
\begin{pmatrix}
	C_1^{\top} + S_1{\varGamma}_{\lambda}\\ C_2^{\top} + S_2{\varGamma}_{\lambda} 
\end{pmatrix}^{\top}
\left[
I+
\begin{pmatrix}
	\varSigma_\lambda & 0\\
	0 & \varGamma_\lambda
\end{pmatrix}
\begin{pmatrix}
	N_1 & 0\\
	0 & N_2
\end{pmatrix}
\right]^{-1}
\begin{pmatrix}
	\varSigma_\lambda & 0\\
	0 & \varGamma_\lambda
\end{pmatrix}
\begin{pmatrix}
	C_1^{\top} + S_1{\varGamma}_{\lambda}\\ C_2^{\top} + S_2{\varGamma}_{\lambda} 
\end{pmatrix}\\
&\quad+
\begin{pmatrix}
B^{\top} + S_3{\varGamma}_{\lambda}
\end{pmatrix}^{\top}
R^{-1}
\begin{pmatrix}
	B^{\top} + S_3{\varGamma}_{\lambda}
\end{pmatrix}. \\
\end{aligned}
\]
Integrating the above equation from $t$ to $T$, and letting $\lambda \rightarrow \infty$,
we obtain the following result by the dominated convergence theorem
\[
\varGamma(t) = -\int_{t}^{T}(
A \varGamma + \varGamma A^{\top} - \mathcal{B}_{\varGamma}R^{-1} \mathcal{B}_{\varGamma}^{\top} - \mathcal{C}_{\varGamma}{\mathcal{N}_{\varGamma}}^{-1} \varGamma \mathcal{C}_{\varGamma}^{\top}	
) \,\mathrm{d}s.
\]
Therefore, $\varGamma(\cdot) \in C([0,T];\mathbb{S}_{+}^n)$ is a solution to \eqref{DECOUPLING:Riccati Equation}.

\medskip

{\bf Step 3: The uniqueness of the solution to equation \eqref{DECOUPLING:Riccati Equation}.}

\medskip

Assume that there exists another solution $\widetilde{\varGamma}(\cdot)$ and set $\varDelta(\cdot) = \varGamma(\cdot) - \widetilde{\varGamma}(\cdot)$, then 
\[
\begin{aligned}
	\dot{\varDelta}(t) =&\ A\varDelta + \varDelta A^{\top} - \mathcal{B}_{\widetilde{\varGamma}} R^{-1} S_3 \varDelta - \varDelta S_3^{\top}R^{-1}\mathcal{B}_{\varGamma}^{\top} - \varDelta S_2 \mathcal{N}_{\varGamma}^{-1} \varGamma \mathcal{C}_{\varGamma}^{\top} \\
	&- \mathcal{C}_{\widetilde{\varGamma}}\Big[\mathcal{N}_{\varGamma}^{-1}\varGamma\mathcal{C}_{\varGamma}^{\top} -  \mathcal{N}_{\widetilde{\varGamma}}^{-1}\widetilde{\varGamma}\mathcal{C}_{\widetilde{\varGamma}}^{\top}\Big] \\
	=&\ A\varDelta + \varDelta A^{\top} - \mathcal{B}_{\widetilde{\varGamma}} R^{-1} S_3 \varDelta - \varDelta S_3^{\top}R^{-1}\mathcal{B}_{\varGamma}^{\top} - \varDelta S_2 \mathcal{N}_{\varGamma}^{-1} \varGamma \mathcal{C}_{\varGamma}^{\top} \\
	&+ \mathcal{C}_{\widetilde{\varGamma}} \mathcal{N}_{\varGamma}^{-1} \varDelta N_2 \mathcal{N}_{\widetilde{\varGamma}}^{-1} \varGamma\mathcal{C}_{\varGamma}^{\top} - \mathcal{C}_{\widetilde{\varGamma}} \mathcal{N}_{\widetilde{\varGamma}}^{-1}\Big[\varDelta\mathcal{C}_{\varGamma}^{\top} + \widetilde{\varGamma} S_2 \varDelta\Big] \\
	=:&\ f(t,\varDelta(t)).
\end{aligned}
\]
Note that $\varDelta(T) = 0$ and $f(t,x)$ is Lipschitz continuous in $x$. A standard argument with the Gronwall inequality shows that $\varDelta(\cdot) = 0$, 
thereby establishing the uniqueness.

Based on the above proof, we have established the unique solvability of the matrix-valued differential equation \eqref{DECOUPLING:Riccati Equation}. Finally, since $\mathcal{N}_{\varGamma}(\cdot)^{-1}$ is bounded on $[0,T]$,
the unique solvability of BSDE with filtering~\eqref{DECOUPLING:adjoint BSDE} can be established in a similar approach to Lemma 4.1 in Wang et al.~\cite{PLQ8}. The detail are omitted for brevity.
This completes the proof.
\end{proof}
\subsubsection{Representation of the optimal control and the value function}

\begin{Theorem} \label{thm:OPTIMAL CONTROL}
	Let (H1) -- (H3) and \eqref{simplification} hold, and let $\xi \in L^2_{\mathcal{F}_T}(\Omega;\mathbb{R}^n)$ be given. Denote $\varGamma(\cdot)$ as the 
solution to the matrix-valued differential equation \eqref{DECOUPLING:Riccati Equation} and $(\varphi(\cdot), \beta_1(\cdot), \beta_2(\cdot))$ as the solution to BSDE with filtering \eqref{DECOUPLING:adjoint BSDE}.
Then the following SDE with filtering 
\begin{equation} \label{FINAL:SDE}
	\begin{cases}\begin{aligned}
		\mathrm{d} X(t) &=  \big[ -A(t)^{\top}X(t) + \widetilde{A}(t) \widehat{X}(t) + b(t) \big] \,\mathrm{d}t \\
		&\quad+\big[ -C_1(t)^{\top}X(t) + \widetilde{C}_1(t)\widehat{X}(t) +c_1(t) \big] \,\mathrm{d}W_1(t)\\ 
		&\quad+\big[ -C_2(t)^{\top}X(t) + \widetilde{C}_2(t) \widehat{X}(t) + c_2(t) \big] \,\mathrm{d}W_2(t), \\
		X(0)&=  0,
	\end{aligned}\end{cases}
\end{equation}
where
\begin{equation}\label{FINAL:Coe}
  \begin{aligned}
  	&\widetilde{A}= S_2^{\top} \mathcal{N}_{\varGamma}^{-1} \varGamma \mathcal{C}_{\varGamma}^{\top} + S_3^{\top} R^{-1} \mathcal{B}_{\varGamma}^{\top}, \quad \widetilde{C}_1=- S_1\varGamma, \quad \widetilde{C}_2 =N_{2}\mathcal{N}_{\varGamma}^{-1} \varGamma \mathcal{C}_{\varGamma}^{\top} - S_2\varGamma, \\
  	&b=- (S_2^{\top} \mathcal{N}_{\varGamma}^{-1} \varGamma S_2 + S_3^{\top} R^{-1} S_3) \widehat{\varphi} + S_1^{\top} \beta_1 + S_2^{\top}\beta_2 + S_2^{\top} (\mathcal{N}_{\varGamma}^{-1} - I)\widehat{\beta}_2, \\
  	&c_1=S_1\varphi + N_{1}\beta_1, \quad c_2=S_2\varphi- N_{2}\mathcal{N}_{\varGamma}^{-1} \varGamma S_2 \widehat{\varphi} + N_{2}\beta_2 + N_{2}(\mathcal{N}_{\varGamma}^{-1} - I)\widehat{\beta}_2,\\
  \end{aligned}
\end{equation}
admits a unique solution $X(\cdot) \in S^2_{\mathbb{F}}(0,T;\mathbb{R}^n)$.
Moreover, the unique optimal control of Problem (BSLQ-P) is given by
\begin{equation} \label{FINAL:OPTIMAL CONTROL}
	u(\cdot) = R(\cdot)^{-1}\big[\mathcal{B}_{\varGamma}(\cdot)^{\top}\widehat{X}(\cdot) - S_3(\cdot)\widehat{\varphi}(\cdot)\big].
\end{equation}
\end{Theorem}

\begin{proof}
For the above SDE with filtering, by Theorem 8.1 in Liptser and Shiyayev~\cite{LS77}, we have
\begin{equation} \label{FINAL:SDE_HAT}
\begin{cases}
	\begin{aligned}
		\mathrm{d}\widehat{X}(t) &= \big[\widehat{A}(t) \widehat{X}(t) + \widehat{b}(t) \big] \,\mathrm{d}t +\big[\widehat{C}_2(t)\widehat{X}(t) + \widehat{c}_2(t) \big] \,\mathrm{d}W_2, \\
		\widehat{X}(0) &=  0,
	\end{aligned}
\end{cases}
\end{equation}
where
\begin{equation}\label{SDE_HAT:Coe}
 \begin{aligned}
 	&\widehat{A}=\widetilde{A}-A^{\top}, \quad
 	\widehat{C}_2=\widetilde{C}_2-C_2^{\top}, \\
 	&\widehat{b}=-(S_2^{\top} \mathcal{N}_{\varGamma}^{-1} \varGamma S_2 + S_3^{\top} R^{-1} S_3) \widehat{\varphi}+S_1^{\top} \widehat{\beta}_1 + S_2^{\top} \mathcal{N}_{\varGamma}^{-1}\widehat{\beta}_2,\\
 	&\widehat{c}_2=(I - N_{2}\mathcal{N}_{\varGamma}^{-1} \varGamma) S_2 \widehat{\varphi}+ N_{2}\mathcal{N}_{\varGamma}^{-1}\widehat{\beta}_2.
 \end{aligned}
 \end{equation}
Under the given conditions, $\widehat{A}(\cdot), \widehat{C}_2(\cdot) \in L^{\infty}(0,T;\mathbb{R}^n)$ and $\widehat{b}(\cdot), \widehat{c}_2(\cdot) \in L^2_{\mathbb{G}}(0,T;\mathbb{R}^n)$.
It follows from the classical theory of SDE that \eqref{FINAL:SDE_HAT} admits a unique solution $\widehat{X}(\cdot) \in S^2_{\mathbb{G}}(0,T;\mathbb{R}^n)$. Then, we consider the following SDE
\begin{equation} \label{Auxiliary:SDE}
\begin{cases}\begin{aligned}
\mathrm{d} x(t) &=  \big[ -A(t)^{\top}x(t) + \widetilde{A}(t) \widehat{X}(t) + b(t) \big] \,\mathrm{d}t \\
		&\quad+\big[ -C_1(t)^{\top}x(t) + \widetilde{C}_1(t)\widehat{X}(t) +c_1(t) \big] \,\mathrm{d}W_1(t)\\ 
		&\quad+\big[ -C_2(t)^{\top}x(t) + \widetilde{C}_2(t) \widehat{X}(t) + c_2(t) \big] \,\mathrm{d}W_2(t), \\
		x(0)&=  0,
\end{aligned}\end{cases}
\end{equation}
where the coefficients are given by \eqref{FINAL:Coe}. Clearly, \eqref{Auxiliary:SDE} admits a unique solution $x(\cdot) \in S^2_{\mathbb{F}}(0,T;\mathbb{R}^n)$, and $\widehat x(\cdot)=\mathbb E[x(\cdot)|\mathcal G_{\cdot}]$ is the unique solution to the following SDE
\begin{equation} \label{Auxiliary:SDE_HAT}
\begin{cases}
	\begin{aligned}
		\mathrm{d}\widehat{x}(t) &= \big[-A(t)^{\top}\widehat{x}(t)+\widetilde{A}(t)\widehat{X}(t) + \widehat{b}(t) \big] \,\mathrm{d}t \\
		&\quad+\big[-C_2(t)\widehat{x}(t)+\widetilde{C}_2(t)\widehat{X}(t) + \widehat{c}_2(t) \big] \,\mathrm{d}W_2, \\
		\widehat{x}(0) &=  0,
	\end{aligned}
\end{cases}
\end{equation}
where $\widehat b(\cdot)$ and $\widehat c_2(\cdot)$ are defined in \eqref{SDE_HAT:Coe}. Comparing \eqref{FINAL:SDE_HAT} with \eqref{Auxiliary:SDE_HAT} and using \eqref{SDE_HAT:Coe}, we deduce that $\widehat X(\cdot)=\widehat x(\cdot)$, and hence \eqref{FINAL:SDE} has a unique solution $X(\cdot) \in S^2_{\mathbb{F}}(0,T;\mathbb{R}^n)$.

Define
\begin{equation}\label{Decouple}
\begin{cases}
	Y = -\varGamma\widehat{X} + \varphi, \\
	Z_1  = {\beta}_1, \\
	Z_2  = \mathcal{N}_{\varGamma}^{-1} \Big( \varGamma \mathcal{C}_{\varGamma}^{\top} \widehat{X} - \varGamma S_2\widehat{\varphi} + {\widehat{\beta}}_2 \Big) + {\beta}_2 - {\widehat{\beta}}_2, 
\end{cases}
\end{equation}
together with \eqref{FINAL:OPTIMAL CONTROL}, SDE~\eqref{FINAL:SDE} can be rewritten as
\[
\begin{cases}\begin{aligned}
		\mathrm{d} X(t) &=  \big[ -A(t)^{\top}X(t) + S_1(t)^{\top}Z_1(t)+S_2(t)^{\top}Z_2(t)+S_3(t)^{\top}u(t) \big] \,\mathrm{d}t \\
		&\quad+\big[ -C_1(t)^{\top}X(t) + S_1(t)Y(t)+N_1(t)Z_1(t) \big] \,\mathrm{d}W_1(t)\\ 
		&\quad+\big[ -C_2(t)^{\top}X(t) + S_2(t)Y(t)+N_2(t)Z_2(t) \big] \,\mathrm{d}W_2(t), \\
		X(0)&=  0.
\end{aligned}\end{cases}
\]
Applying Itô's formula to $Y$ yields
\[
\begin{aligned}
\mathrm{d}Y &= -\dot{\varGamma} \widehat{X} \,\mathrm{d}t-\varGamma \,\mathrm{d}\widehat{X} + \mathrm{d}\varphi \\
&= \Big[\big(- A\varGamma + BR^{-1}\mathcal{B}_\varGamma^{\top} + C_2\mathcal{N}_\varGamma^{-1}\varGamma\mathcal{C}_\varGamma^{\top}\big)\widehat{X} + A\varphi - \big(BR^{-1}S_3 + C_2\mathcal{N}_\varGamma^{-1}\varGamma S_2\big)\widehat{\varphi}\\ 
&\quad+ C_1\beta_1 + C_2\beta_2 + C_2\big(\mathcal{N}_\varGamma^{-1}-I\big)\widehat{\beta}_2\Big] \,\mathrm{d}t + \beta_1 \,\mathrm{d}W_1\\
&\quad+ \Big[ \mathcal{N}_{\varGamma}^{-1} \big( \varGamma \mathcal{C}_{\varGamma}^{\top} \widehat{X} - \varGamma S_2\widehat{\varphi} + {\widehat{\beta}}_2 \big) + {\beta}_2 - {\widehat{\beta}}_2 \Big] \,\mathrm{d}W_2,
\end{aligned} 
\]
from which it follows that $(Y(\cdot),Z_1(\cdot),Z_2(\cdot))$ defined by \eqref{Decouple} indeed satisfies \eqref{BSLQ:SE}.
Moreover, we obtain from \eqref{FINAL:OPTIMAL CONTROL} that
\[
S_3(\cdot)\widehat{Y}(\cdot) - B^{\top}(\cdot)\widehat{X}(\cdot) + R(\cdot)u(\cdot) = 0.
\]
Thus, $(X(\cdot),Y(\cdot),Z_1(\cdot),Z_2(\cdot),u(\cdot))$ satisfies the Hamiltonian system \eqref{FBSDE:Hamilton System}. Theorem~\ref{theorem: FBSDE} ensures that $u(\cdot)$ defined by \eqref{FINAL:OPTIMAL CONTROL} is the unique optimal control of Problem (BSLQ-P).
\end{proof}

\begin{Theorem}
	Let (H1) -- (H3) and \eqref{simplification} hold. The value function of Problem (BSLQ-P) is given by
\[
\begin{aligned}
V(\xi) &= \mathbb{E}\int_{0}^{T}\Big\{-\Big\langle (S_2^{\top}\mathcal{N}_{\varGamma}^{-1}\varGamma S_2 + S_3^{\top}R^{-1}S_3)\widehat{\varphi} - 2S_2^{\top}(\mathcal{N}_{\varGamma}^{-1} - I)\widehat{\beta}_2,\, \widehat{\varphi}\Big\rangle \\
&\quad + 2\Big\langle S_1^{\top}\beta_1 + S_2^{\top}\beta_2,\, \varphi\Big\rangle + \langle N_{1}\beta_1,\, \beta_1\rangle + \langle N_{2}\beta_2,\, \beta_2\rangle + \Big\langle N_{2}(\mathcal{N}_{\varGamma}^{-1} - I)\widehat{\beta}_2,\, \widehat{\beta}_2 \Big\rangle\Big\} \,\mathrm{d}t.
\end{aligned}
\]
where $\varGamma(\cdot)$ is the solution to the matrix-valued differential equation \eqref{DECOUPLING:Riccati Equation} and 
$(\varphi(\cdot), \beta_1(\cdot), \beta_2(\cdot))$ is the solution to BSDE with filtering \eqref{DECOUPLING:adjoint BSDE}.
\end{Theorem}

\begin{proof}
Let $(X(\cdot),Y(\cdot),Z_1(\cdot),Z_2(\cdot),u(\cdot))$ be the solution to Hamiltonian system \eqref{FBSDE:Hamilton System}. On the one hand, applying Itô's formula to $\langle X(\cdot), \, Y(\cdot) \rangle$, we get
\[\begin{aligned} 
	\mathbb{E}\langle X(T),\, Y(T)\rangle  &=\mathbb{E} \int_{0}^{T}\Big[ \langle X,\, AY + Bu + C_1Z_1 + C_2Z_2\rangle   + \langle -A^{\top} + S_1^{\top}Z_1 + S_2^{\top}Z_2 + S_3^{\top}u,\, Y\rangle \\
	&\quad +\langle -C_1^{\top}X + S_1Y + N_{1}Z_1,\, Z_1\rangle   + \langle -C_2^{\top} + S_2Y + N_{2}Z_2,\, Z_2\rangle   \Big] \,\mathrm{d}t. \\
	&=\mathbb{E} \int_{0}^{T}\Big[ \langle S_1^{\top}Z_1 + S_2^{\top}Z_2 + S_3^{\top}u,\, Y\rangle   + \langle S_1Y + N_{1}Z_1,\, Z_1\rangle + \langle S_2Y + N_{2}Z_2,\, Z_2\rangle \\
	&\quad + \langle B^{\top}X,\, u\rangle  \Big] \,\mathrm{d}t.
\end{aligned}\]
On the other hand, from the definitions of the cost functional and the value function,
\[\begin{aligned}
V(\xi) &= J(\xi;u(\cdot)) \\
&= \mathbb{E} \int_{0}^{T}\big[\langle  Ru,\, u \rangle   + \langle N_{1}Z_1,\, Z_1\rangle   + \langle N_{2}Z_2,\, Z_2\rangle + 2\langle S_1Y,\, Z_1\rangle \\
&\quad+ 2\langle S_2Y,\, Z_2\rangle   + 2\langle S_3Y,\, u\rangle  \big] \,\mathrm{d}t\\
&=\mathbb{E} \int_{0}^{T}\Big[ \langle S_1^{\top}Z_1 + S_2^{\top}Z_2 + S_3^{\top}u,\, Y\rangle   + \langle S_1Y + N_{1}Z_1,\, Z_1\rangle + \langle S_2Y + N_{2}Z_2,\, Z_2\rangle\\
&\quad + \langle S_3Y + Ru,\, u\rangle \Big] \,\mathrm{d}t, \\
\end{aligned}\]
and note that
\[ 
\mathbb{E} \int_{0}^{T}\langle S_3Y + Ru,\, u\rangle \,\mathrm{d}t   = \mathbb{E} \int_{0}^{T}\langle S_3\widehat{Y} + Ru,\, u\rangle \,\mathrm{d}t   = \mathbb{E} \int_{0}^{T}\langle B^{\top}\widehat{X}, u\rangle \,\mathrm{d}t.
\]
Therefore, based on the above two aspects and $\varphi(T) = \xi = Y(T)$, it follows that
\[
 V(\xi) = \mathbb{E}\langle X(T),\, Y(T)\rangle = \mathbb{E}\langle X(T),\, \varphi(T)\rangle.
\]
Moreover, Theorem~\ref{thm:OPTIMAL CONTROL} shows that $X(\cdot)$ also satisfies \eqref{FINAL:SDE}. Applying Itô's formula to $\langle X(\cdot), \, \varphi(\cdot) \rangle$, we finally obtain the desired result. The proof is complete.
\end{proof}

\begin{Remark}
By employing transformation \eqref{Transformations}, we equivalently transform the original problem into solving one under condition \eqref{simplification}, which leads to the above results. Corresponding to \eqref{Transformations}, we introduce the notations
\[
\begin{aligned}
\mathcal{N}_{\varGamma}^\varPhi = I + \varGamma N_{2}^\varPhi, \quad\mathcal{B}_{\varGamma}^\varPhi = B + \varGamma (S_3^\varPhi)^{\top}, \quad\mathcal{C}_{\varGamma}^\varPhi = C_2 + \varGamma (S_2^\varPhi)^{\top}.	
\end{aligned}
\]
Then, with the argument in Section \ref{section:simplification} and the results in Section \ref{section:CONSTRUCT of the optimal control}, the conclusions can be easily extended to the original problem without \eqref{simplification}. Since the results are analogous, the details are omitted for brevity.
\end{Remark}

\section{Conclusions} \label{sec4}

In summary, we have studied an indefinite BSLQ optimal control problem with partial information. 
Our work fundamentally relies on the assumption that the cost functional is uniformly convex. We derive a Hamiltonian system, which is an FBSDE with filtering. 
To decouple this system, we furthe introduce a matrix-valued differential equation and a BSDE with filtering, both of which play crucial roles 
in the construction of the optimal control. To prove their solvability, inspired by \cite{BLQ5} and \cite{ILQ6}, we explore the relationship 
between forward and backward problems. Specifically, we show that the uniform convexity of the cost functional in the backward problem implies the uniform convexity of the cost functionals in a family of forward problems (see Theorem~\ref{thm:bfr}). 
Based on this, along with the solvability of the Riccati equations associated with this family of forward problems under the uniform convexity assumption (see Corollary~\ref{coro2}), 
we then prove the solvability of the matrix-valued differential equation for the backward problem
by taking limit, and consequently obtain the existence and uniqueness of the solution to the BSDE with filtering (see Theorem~\ref{thm:bRiccati}).
We provide explicit forms of optimal control and value function at the end of the paper. 
In our study, partial information brings filtering to the Hamiltonian system and affects the form of the matrix-valued differential equation.
Additionally, there is another type of incomplete information known as partial observation, which is more complicated.
We plan to investigate the corresponding indefinite stochastic LQ problem in future work.

\vspace{6pt} 


\end{document}